\documentclass[11pt,reqno]{amsart}
\usepackage{amsmath}
\usepackage{}
\usepackage{}
\usepackage{amsfonts}
\usepackage{a4wide}
\usepackage{amsmath,amsthm,amssymb,amscd}
\usepackage{latexsym}
\usepackage[numbers,sort&compress]{natbib}
\allowdisplaybreaks
\numberwithin{equation}{section}

\newtheorem{theorem}{Theorem}[section]
\newtheorem{proposition}[theorem]{Proposition}
\newtheorem{corollary}[theorem]{Corollary}
\newtheorem{lemma}[theorem]{Lemma}

\theoremstyle{definition}

\newtheorem{remark}[theorem]{Remark}

\newcommand{\ds}{\displaystyle}

\newcommand{\hv}{{H_{V,\varepsilon}^s(\mathbb{R}^N)}}

\newcommand{\eq}{\eqref}
\def\r{\ref}

\def\c{\cite}

\def\k{\kappa}

\newcommand{\Ds}{{\dot{H}^s(\mathbb{R}^N)}}

\def\P{\mathcal P_\varepsilon}

\def\Fs{(-\Delta)^s}

\def\d{\mathrm{d}}
\def\rn{\mathbb{R}^N}
\def\R{\mathbb R}
\def\N{\mathbb N}

\def\wq{\infty}

\def\a{\alpha}
\def\be{\beta}
\def\ga{\gamma}

\def\la{\lambda}
\def\si{\sigma}

\def\var{\varphi}
\def\om{\omega}

\def\Om{\Omega}
\def\De{\Delta}

\newcommand{\dist}{{\rm dist}}
\newcommand{\va}{\varepsilon}

\usepackage{color}

\newcommand{\La}{\Lambda}

\begin{document}

\title[Fractional Schr\"{o}dinger equations with decaying potentials]
{Existence and decays  of solutions for fractional Schr\"{o}dinger  equations  with  decaying potentials}

\thanks {The research was supported  by the Natural Science Foundation of China  (No. 12271196, 11931012).}
\author[ Y. Deng, S. Peng, X. Yang, ]{Yinbin Deng \textsuperscript{1}, Shuangjie Peng\textsuperscript{2} and Xian Yang \textsuperscript{3}}

{
%\author{Yinbin Deng}
\footnotetext[1]{School of Mathematics and Statistics \& Hubei Key Laboratory of Mathematical Sciences,
Central China Normal University,
Wuhan 430079, P. R. China. Email: ybdeng@ccnu.edu.cn.}
%\author{Shuangjie Peng}
\footnotetext[2]{School of Mathematics and Statistics  \& Hubei Key Laboratory of Mathematical Sciences,
Central China Normal University,
Wuhan 430079, P. R. China. Email: sjpeng@ccnu.edu.cn.}
%\author{Xian Yang}
\footnotetext[3]{ School of Mathematics and Statistics \& Hubei Key Laboratory of Mathematical Sciences,
Central China Normal University,
Wuhan 430079, P. R. China.   Email: yangxian@mails.ccnu.edu.cn.}
}

\begin{abstract}
   We revisit the following  fractional Schr\"{o}dinger   equation
   \begin{align}\label{1a}
    \varepsilon^{2s}(-\Delta)^su +Vu=u^{p-1},\,\,\,u>0,\ \ \ \mathrm{in}\ \R^N,
   \end{align}
where $\varepsilon>0$ is a small parameter, $(-\Delta)^s$ denotes the fractional Laplacian,  $s\in(0,1)$, $p\in (2, 2_s^*)$, $2_s^*=\frac {2N}{N-2s}$, $N>2s$, $V\in C\big(\R^N, [0, +\infty)\big)$ is a potential.  Under various  decay assumptions on $V$, we introduce a unified  penalization argument combined with a comparison principle and iteration process to detect an explicit threshold value $p_*$, such that the above problem admits positive concentration solutions if $p\in (p_*, \,2_s^*)$, while it has no positive weak solutions for $p\in (2,\,p_*)$ if $p_*>2$, where the threshold $p_*\in [2, 2^*_s)$ can be characterized explicitly by
\begin{equation*}\label{qdj111}
p_*=\left\{\begin{array}{l}
 2+\frac {2s}{N-2s} \ \ \ \text { if } \lim\limits_{|x| \to \infty} (1+|x|^{2s})V(x)=0,\vspace{1mm} \\
 2+\frac {\omega}{N+2s-\omega} \text { if }  0<\inf (1+|x|^\omega)V(x)\le \sup (1+|x|^\omega)V(x)< \infty \text { for some } \omega \in [0, 2s],\vspace{1mm}\\
 2 \ \ \ \ \ \ \ \ \ \ \ \ \ \  \text { if } \inf V(x)\log(e+|x|^2)>0.
\end{array}\right.
\end{equation*}
Moreover, corresponding to the various decay assumptions of $V(x)$, we obtain  the decay properties of the solutions at infinity.% including  the lower and upper decay rate.

{\bf Key words:} Fractional Schr\"{o}dinger  equations; penalization method; decaying potentials; positive solutions; decay estimates

{\bf AMS Subject Classifications:} 35J15, 35A15, 35J10.
\end{abstract}

\maketitle
\section{Introduction}\label{s1}
In this paper, we study the following  nonlinear fractional Schr\"{o}dinger  equation
\begin{equation}\label{eqs1.1}
\varepsilon^{2s}(-\Delta)^su +Vu=u^{p-1},\,\,\,u>0, \ \ \  \mathrm{in}\ \R^N,
\end{equation}
 where~$\varepsilon>0$, $N>2s$, $s\in(0,1)$,  $2<p<2_s^*:=\frac{2N}{N-2s}$, $V\in C\big({\R}^N, [0, \infty)\big)$ is an external potential,  the fractional Laplacian $(-\Delta)^s$, up to a normalization constant, is defined by
\begin{equation*}\label{eqs1.2}
(-\Delta)^su(x):=2P.V.\int_{\mathbb{R}^N}\frac{u(x)-u(y)}{|x-y|^{N+2s}}~\d y=2\lim\limits_{r\to 0}\int_{\mathbb{R}^N\backslash B_r(x)}\frac{u(x)-u(y)}{|x-y|^{N+2s}}~\d y.
\end{equation*}
 Problem $\eq{eqs1.1}$ arises from finding standing waves (i.e., $\psi(x,t)=e^{-iEt/\va}u(x)$) to the time-dependent fractional Schr\"{o}dinger equation
\begin{equation}\label{e2m}
i\va\frac{\partial \psi}{\partial t}=\varepsilon^{2s}(-\Delta)^s\psi +(V+E)\psi-|\psi|^{p-2}\psi, \ \ \  (x,t)\in \R^N\times \R_+,
\end{equation}
which was introduced by Laskin in \cite{las}  as a fundamental equation modeling  fractional Quantum Mechanics in view of a path integral over the L\'{e}vy flights paths.
Problem \eq{e2m} admits extensive applications in modeling diverse physical phenomena, especially for dealing with relativistic particles ($s = 1/2$).
See \cite{clp,ege2012} and the references therein for more backgrounds on \eq{e2m}.

When $\varepsilon>0$ is a small parameter, which is typically related to the Planck constant, from the physical prospective \eq{eqs1.1} is particularly important, since its solutions  as $\varepsilon\to 0$ are called semi-classical bound states. Physically, it is expected that in the semi-classical limit $\varepsilon\to 0$ there should be a correspondence between solutions of the equation \eq{eqs1.1} and critical points of the potential $V$, which governs the classical dynamics.

Recently, problem \eq{eqs1.1} has been extensively investigated. When $\varepsilon=1$ and $V\equiv \lambda>0$, it was proved in \cite{paj2012} that \eqref{eqs1.1} has a positive ground state, which was subsequently verified  by  Frank et al.  in \cite{Frank.L-CPAM-2016} to be non-degenerate and unique up to translation. When $\varepsilon\to 0$, inspired by the penalized idea introduced  in \cite{mp1996},   Alves and Miyagaki  showed that $\eq{eqs1.1}$ has a family of solutions concentrating at a local minimum of $V$ in the case $\inf_{x\in\rn}V(x)>0$  in \cite{co2016}.
For more results in this case, one can refer \cite{vab2017,jmj2014,ss2013} and the references therein.

If the external potential $V$ tends to zero at infinity (particularly $V$ is compactly supported), the situation is much more complicated since the action functional corresponding to \eq{eqs1.1} is typically not well defined nor Fr\'echet differentiable
on $H_{V, \varepsilon}^s(\R^N)$ (which is defined later). Even in the classical case $s=1$, this difficulty is not only technical. Compared with the classical case $s=1$, where the solutions usual decay exponentially, the solutions of the nonlocal problem \eqref{eqs1.1}  with $s\in (0,\,1)$ decay polynomially at infinity as we will see later.

Recently, problem \eqref{eqs1.1} with   potential vanishing at infinity was considered in \cite{xsc2019,xly2019}. For $2+\frac{2s}{N-2s}<p<2_s^*$, when $V(x)$ achieves its local minimum in a bounded domain and satisfies the following decay assumption
$$(\mathcal{V}_*)\,\,\ds \ \  \inf_{x\in\rn}V(x)|x|^{2s}>0,$$
it was proved in \cite{xsc2019} that problem \eqref{eqs1.1} has positive solutions for $\varepsilon>0$ sufficiently small. A natural question is  whether the restriction $p>2+\frac{2s}{N-2s}$ is optimal with respect to the decay assumption on the potential $V$.
Moreover, what about the existence and non-existence of the positive solutions if $V(x)$ decays to zero at infinity with various decay behaviors?

The aim in the present paper is three-fold: firstly, we try to detect a threshold  $p_*\in [2, 2^*_s)$ such that problem (\ref {eqs1.1}) exists  concentration solutions $u_\varepsilon$ as $\varepsilon\to 0$ if $p\in (p_*, 2^*_s)$, while it has no positive weak solutions for all $\varepsilon>0$ and $p\in (2, p_*)$ if $p_*>2$. Secondly, as far as we know, for various conditions on  $V(x)$ at infinity, the arguments used to deal with problem (\ref {eqs1.1}) are usually different. Hence, it is interesting that  we introduce a united framework under which problem (\ref {eqs1.1}) with $V(x)$ having various behaviors at infinity, including $V(x)\ge C>0$ in $\rn$ and $V(x)\to0$ at various rates at infinity, can be dealt with uniformly. Lastly, corresponding to various behaviors of $V(x)$ at infinity, we intend to explore the exact decay rates  of the solutions at infinity.

At first, we study the existence of positive solution $u_\varepsilon$. We  assume  that $V$ satisfies the following local assumption %($\mathcal{V}$)
:\\
($\mathcal{V}$) $V\in C(\rn,[0,+\wq))$, and there exists a bounded open set~$\Lambda \subset \R^N$~such that
$$
0<V_0=\underset{x\in \Lambda}{\rm inf}V(x)<\underset{x\in \partial\Lambda}{\rm min} V(x).
$$

Without loss of generality, we can assume   that $0\in\Lambda$. %and $\partial \La$ is smooth
To discuss the influence of decay property of $V$ on the existence and decay rate of solutions to problem (\ref {eqs1.1}), we also assume the following  various decays  on potential $V$:

$(\mathrm{\mathcal{V}_{fd}})$ {\bf (fast decay)}:
$$\ds \lim_{|x| \to \infty} (1+|x|^{2s})V(x)=0;$$

$(\mathrm{\mathcal{V}_{sd}})$ {\bf (slow decay or no decay)}:
  $$
  \ds 0<\inf_{x\in\rn} (1+|x|^\omega)V(x)\le \sup_{x\in\rn} (1+|x|^\omega)V(x)< +\infty \ \text{for  some} \ \omega \in [0, 2s];
  $$

$(\mathrm{\mathcal{V}_{lsd}})$ {\bf (lower slow decay  or no decay)}:
$$\ds \inf_{x\in\rn} (1+|x|^\omega)V(x)>0 \text { for some } \omega \in [0, 2s];$$

$(\mathrm{\mathcal{V}_{usd}})$ {\bf (upper slow decay)}:
$$\ds  \sup_{x\in\rn} (1+|x|^\omega)V(x)< +\infty \text { for some } \omega \in [0, 2s];$$

$(\mathrm{\mathcal{V}_{log}})$ {\bf (lower logarithmic decay)}:
$$\ds  \inf_{x\in\rn} V(x)\log(e+|x|^2)>0.$$

We say that $u$ is a weak solution to the equation $\eq{eqs1.1}$
if $u\in\hv$ satisfies
\begin{equation*}\label{445}
\varepsilon^{2s}\iint_{\R^{2N}}\frac{\big(u(x)-u(y)\big)\big(\var(x)-\var(y)\big)}{|x-y|^{N+2s}}
+\int_{\R^N}
V(x)u\var=\int_{\rn}|u|^{p-2}u\var,\quad \forall\ \var\in \hv,
\end{equation*}
 where
 $$
H_{V, \varepsilon}^s(\R^N):=\Big\{u\in L^{2_s^*}(\R^N): \iint_{\R^{2N}}\frac{|u(x)-u(y)|^2}{|x-y|^{N+2s}}+\int_{\rn}V(x)u^2<\infty \Big\}.
$$

%For convenience, hereafter, given $U\subset\rn$ and $\tau>0$, we denote $C^\tau(U)=C^{[\tau],\tau-[\tau]}(U)$
%with $[\tau]$ denoting the largest integer no larger than $\tau$.

\begin{theorem}\label{thm1.2}
Let $V$ satisfy ($\mathcal{V}$) and $p\in (2, 2_s^*)$ satisfy one of the following three assumptions:

$(\mathrm{\mathcal{Q}_1})$
$p>q_*:= 2+\frac{2s}{N-2s}$;

$(\mathrm{\mathcal{Q}_{2}})$  $p> q_\om:=2+\frac{\om}{N+2s-\om}$ and $(\mathrm{\mathcal{V}_{lsd}})$ {\bf (lower slow decay or no decay)};

$(\mathrm{\mathcal{Q}_{3}})$  $p> 2$ and $(\mathrm{\mathcal{V}_{log}})$ {\bf (lower logarithmic decay)}.

\noindent
Then there exists $\varepsilon_0>0$ such that for  $\varepsilon\in(0,\varepsilon_0)$, problem \eq{eqs1.1} has a positive weak solution $u_{\varepsilon}\in C^\sigma_{\mathrm{loc}}(\rn)\cap L^\wq(\rn)$ with $\sigma\in (0,\min\{2s,1\})$.
Moreover, $u_{\varepsilon}$ has a global maximum point  $x_{\varepsilon}\in \Lambda$  such that
   $$\lim\limits_{\varepsilon\to 0}V(x_{\varepsilon})=V_0~$$
   and
   \begin{equation}\label{yyy}
     u_{\varepsilon}(x)\leq\frac{C_\ga\varepsilon^{\gamma}}{\varepsilon^{\gamma}+|x-x_{\varepsilon}|^{\gamma}}
   \end{equation}
 for a  positive constant $C_\ga$ independent of $\va$, where $\gamma>0$ is a positive constant close to $N-2s$ from below if
 $(\mathrm{\mathcal{Q}_1})$ holds, close to $N+2s-\om$ from below if $(\mathrm{\mathcal{Q}_2})$ holds, and close to $N+2s$ from below if $(\mathrm{\mathcal{Q}_3})$ holds.
\end{theorem}

\begin{remark}Theorem \r{thm1.2}  improve and develop the main result in \cite{xsc2019} by extending
 the permitted range of $p$ for the existence of $u_\varepsilon$ from $(2+\frac{2s}{N-2s},2_s^*)$ to $(2+\frac{2s}{N},2_s^*)$ and refining the decay estimate of $u_\va$.   More interestingly, when $V$ decays logarithmically or $(\mathrm{\mathcal{Q}_{2}})$ with $\om=0$ (i.e., $\inf_{\rn} V(x)>0$) holds, the admissible range of $p$ for the existence of $u_\va$  is exactly $(2,2_s^*)$, which coincides with that of \cite{co2016}  where $\inf_{\rn} V(x)>0$, but our argument here is quite different from that of \cite{co2016}. Indeed, we establish a unified method to deal with
 both the non-vanishing potential and  vanishing potential.

\end{remark}

As a contrast to Theorem \r{thm1.2}, we have the following sharp nonexistence results.
\begin{theorem}\label{thm1.3}
Let $V\in C(\rn,[0,+\wq))$. For any $\va>0$, \eq{eqs1.1} has no positive weak solutions if one of the following two assumptions holds:

$(\mathrm{\mathcal{Q}'_1})$ $2<p<q_*= 2+\frac{2s}{N-2s}$ and $(\mathrm{\mathcal{V}_{fd}})$ {\bf (fast decay)};

$(\mathrm{\mathcal{Q}'_2})$ $2<p<q_\om= 2+\frac{w}{N+2s-\om}$ and $(\mathrm{\mathcal{V}_{usd}})$ with $\omega >0$ {\bf (upper slow decay)}.
\end{theorem}

Combining Theorem  \r{thm1.2} with Theorem \r{thm1.3}, we obtain the following corollary:
\begin{corollary}\label{c3011} Suppose that $V$ satisfies ($\mathcal{V}$) and one of the decay assumptions $(\mathrm{\mathcal{V}_{fd}})$ {\bf (fast decay )}and
$(\mathrm{\mathcal{V}_{sd}})$ {\bf (slow decay or no decay)} holds. Then there exists a threshold  $p_*\in [2, 2^*_s)$ such that problem (\ref {eqs1.1}) admits a positive concentration solution $u_\varepsilon$  as $\varepsilon\to 0$ if $p\in (p_*, 2^*_s)$ and  has no positive weak solutions for all $\varepsilon>0$ and $p\in (2, p_*)$ if $p_*>2$, where  the threshold value $p_*\in [2, 2^*_s)$ can be characterized by
\begin{equation*}\label{qdj111}
p_*=\left\{\begin{array}{l}
q_*  \ \text { if } \ (\mathrm{\mathcal{V}_{fd}}) \ \text {{\bf (fast decay)} holds,} \\
q_\om \ \text { if } \ (\mathrm{\mathcal{V}_{sd}})  \ \text {{\bf (slow decay)} holds}.
\end{array}\right.
\end{equation*}
\end {corollary}

Note that the local assumption ($\mathcal{V}$) is not necessary in Theorem \r{thm1.3}. Meanwhile,
Corollary  \r{c3011} and Theorem \r{thm1.2} indicate that in the case of slow decay $(\mathrm{\mathcal{V}_{sd}})$, the threshold $p_*=q_\omega$ and the upper decay estimate of $u_\va$  depend on the decay rate $|x|^{-\omega}$ of $V(x)$, while in the case of fast decay $(\mathrm{\mathcal{V}_{fd}})$, they are independent of the decay rate of $V(x)$.  We point out that it remains an open problem  to deal with the threshold  $p_*=2+\frac {2s}{N-2s}$ if $\lim_{|x|\to\wq}(1+|x|^{2s})V(x)=0$ or  $p_*=q_\om  =2+\frac {\om}{N+2s-\om} $ if $\sup_{x\in\rn}(1+|x|^{\om})V(x)<\wq$ with $\om\in(0,2s]$.

%The main idea in the proof of Theorem \ref{thm1.4} is Moser iteration and Schauder estimates (see Lemmas \ref{w}, \ref{I}, \ref{s} and \ref{p} for more details).

Recall that we have given a upper decay estimate of $u_\va$ in  Theorem \r{thm1.2}. The second part of this paper is to provide some priori lower decay estimates for general positive weak solutions of \eq{eqs1.1} under various decay assumptions on potential $V$.
\begin{theorem}\label{thm1.5}
Assume that $V\in C(\rn,[0,+\wq))$  satisfies $(\mathrm{\mathcal{V}_{usd}})$ and  $p\in(p_*,2_s^*)$, where
\begin{equation*}\label{qdj111}
p_*=\left\{\begin{array}{l}
2+\frac {2s}{N-2s} \ \ \ \text { if } \omega \in (2s, +\infty), \\
2+\frac {\omega}{N+2s-\omega} \text { if }  \omega \in [0, 2s].
\end{array}\right.
\end{equation*}
 Then positive weak solutions $u$ of \eq{eqs1.1} satisfy the following lower decay estimates:

(i) $u(x)\ge \frac{C_\va}{1+|x|^{N-2s}}$ if   $\om>(N-2s)(p-2)>2s$;

(ii) $u(x)\ge \frac{C_{\va,\mu}}{1+|x|^{\mu}}$ for any $\mu>N-2s$ if  $\om \in (2s, (N-2s)(p-2)]$;

(iii)  $u(x)\ge \frac{C_\va}{1+|x|^{N+2s-\om}}$ if  $\om\in[0,2s]$.

\noindent Here $C_\va, C_{\va,\mu}>0$ are constants.
\end{theorem}

From Theorem \r{thm1.5} we know that  the local assumption ($\mathcal{V}$) is dispensable and the decays  of positive weak solutions to \eq{eqs1.1} depend on those  of the potential $V$.  In general,  the positive weak solutions of \eq{eqs1.1}  decay slower if the potential $V$ decays faster.

When employing a penalized idea, one looks forward to  using the upper decay estimate of penalized solutions to recover the original problem. From this point of view,
Theorems \r{thm1.3} and \r{thm1.5} are important and helpful to the construction of penalization. For instance, in the case that $V$ has compact support, if one expects  \eq{eqs1.1} to have a positive weak solution $u$ satisfying $0< u(x)\le \frac{\va^\theta}{|x|^{\tau}}$, it is necessary  to consider $p> 2+\frac{2s}{N-2s}$ and $\tau<N-2s$, which matches Theorem \r{thm1.2}.

Due to Theorem \r{thm1.5}, we are able to  characterize precisely the decay property of $u_\va$ given by Theorem \r{thm1.2} if the decay rate of $V$ is specific, i.e.,

$(\mathcal{V}_\om)$ $\frac{C_1}{1+|x|^\om}\le V(x)\le \frac{C_2}{1+|x|^\om}$ in $\rn\backslash\Om$ for some $C_1, C_2\in(0,\wq)$, $\om\in [0,+\wq]$ and bounded open set $\Om\subset\rn$, where $\Om=\La$ if $\om\in[0,+\wq)$ and $\La\subset\subset\Om$ if $\om=+\wq$.

We define  $\frac{1}{1+|x|^\om}=0$ for $\om=+\wq$.

\begin{corollary}\label{c30} Assume ($\mathcal{V}$) and ($\mathcal{V}_\om$), $p\in (p_*, 2_s^*)$, where
\begin{equation*}\label{qdj111}
p_*=\left\{\begin{array}{l}
q_*=2+\frac {2s}{N-2s} \ \ \ \text { if } \omega \in (2s, +\infty], \\
q_\om=2+\frac {\omega}{N+2s-\omega} \text { if }  \omega \in [0, 2s].
\end{array}\right.
\end{equation*}
Let $u_\va$ be given by Theorem \r{thm1.2}. Then there hold

(i) $\frac{A_\va}{1+|x|^{N-2s}}\le u_\va(x)\le \frac{C_{\va,\ga}}{1+|x|^{\ga}}$ for any $\ga<N-2s$ if $\om>(N-2s)(p-2)>2s$;

(ii) $\frac{A_{\va,\mu}}{1+|x|^{\mu}}\le u_\va(x)\le \frac{C_\va}{1+|x|^{N-2s}}$ for any $\mu>N-2s$ if  $\om \in (2s, (N-2s)(p-2))$;

(iii)  $\frac{A_\va}{1+|x|^{N}}\le u_\va(x)\le \frac{C_{\va,\ga}}{1+|x|^{\ga}}$ for any $\ga<N$  if  $\om=2s$;

(iv)  $\frac{A_\va}{1+|x|^{N+2s-\om}}\le u_\va(x)\le \frac{C_\va}{1+|x|^{N+2s-\om}}$ if $\om\in[0,2s)$.

\noindent Here $A_\va, C_\va,  A_{\va,\mu}, C_{\va,\mu}>0$ are constants.
\end{corollary}

\begin{remark}
Without imposing penalization, it is very hard to  get any priori upper decay estimate for positive solutions to \eq{eqs1.1} with $V$ decaying to $0$ at infinity, since we cannot compare the term $Vu$ with  $u^{p-1}$ at infinity, which is quite different  from the case
$\inf_{\rn}V>0$, where $Vu>>u^{p-1}$ at infinity as long as  $u\to0$ as $|x|\to\wq$.
However, as a consequence of some suitable penalization, it is possible to get the upper decay estimate of a particular solution $u$ to \eq{eqs1.1}. More precisely, to expect some upper decay property of candidate positive solutions,  it is natural to modify the nonlinearity $u^{p-1}$ out of $\La$ as
$\chi_{\La^c}\min\{u_+^{p-2},\frac{\va^\theta}{|x|^\tau} \}u_+$. If we successfully  recover the original problem, it holds naturally that
$u_+^{p-2}\le\frac{\va^\theta}{|x|^\tau}$ in $\La^c$, which gives the upper decay property of $u$.
Consequently, the key lies ultimately in whether we can find suitable $\theta$ and $\tau$. In particular, Theorem \r{thm1.3} implies the nonexistence of $\theta$ and $\tau$ for finding positive weak solutions to \eq{eqs1.1} under some conditions.
Additionally, Corollary \r{c30} implies that the lower decay estimates in Theorem \r{thm1.5} are almost optimal.
\end{remark}

Now we explain the  overview of our main ideas to prove the main results.

Normally, it should be preferential to consider the following functional corresponding to $\eq{eqs1.1}$
\begin{equation}\label{zr}
E_\va(v):=\frac{\varepsilon^{2s}}{2}\iint_{\R^{2N}}\frac{|u(x)-u(y)|^2}{|x-y|^{N+2s}}
+\frac{1}{2}\int_{\R^N}
V(x)u^2-\frac{1}{p}\int_{\rn}|u|^p,\ v\in \hv.
\end{equation}
However, $E_\va$ is not well-defined when $V$ decays very fast. For instance,  $\omega_\mu:=\frac{1}{(1+|x|^2)^{\frac{\mu}{2}}}\in \hv$ but $\int_{\rn}|w_\mu|^p=+\wq$   for any $\mu\in(\frac{N-2s}{2},\frac{N}{p})$  if $V\le \frac{C}{1+|x|^{2s}}$. Moreover, it is hard to verify  the (P.S.) condition directly only under the local assumption $(\mathcal{V})$ on $V$. To overcome these difficulties, we introduce some existence results in \cite{xsc2019} on a penalized problem to \eq{eqs1.1} by modifying the nonlinearity (see \eq{qaj} and Lemma \r{lem5.1}). Naturally, the next crucial step is to recover the original problem \eq{eqs1.1} from the penalized problem, which was also done in  \cite[Section IV]{xsc2019}. However, neither the admissible range of $p$ was  optimal, nor the relation between the admissible range of $p$ and the specific decay rate of $V$ was established there. As a result, we will construct more precise super-solutions to the linear problem \eq{q1t}. Different from \cite{xsc2019}, the super-solutions here are of the form
$$w_\mu:=\frac{1}{(1+|x|^2)^{\mu/2}},$$
where $\mu>0$ depends on the  decay rate of $V$. In particular, $\mu$ should be taken larger if $V$ decays slower (see Proposition \r{ed5}).
Though $w_\mu$ is smooth, due to the nonlocal nature of $(-\De)^s$, it is quite difficult to estimate $(-\De)^sw_\mu$. Moreover, delicate nonlocal analysis should also be used to establish the suitable integral properties of $w_\mu$ (see Proposition \r{cc5}).

As another  novelty, we develop some sharp nonexistence results to \eq{eqs1.1} in Theorem \r{thm1.3}. In the proof of Theorem \r{thm1.2} we use crucially the lower decay estimates of $(-\De)^sw_\mu$, while the proof of Theorem \r{thm1.3} depends strongly
on the upper decay estimates of $(-\De)^sw_\mu$.  After establishing a comparison principle, we  implement some delicate iteration processes to prove Theorem \r{thm1.3}. We emphasize here that this is quite different from that of the local case $s=1$ because of the different asymptotic behaviors of $(-\De)^sw_\mu$ and $-\De w_\mu$. For example, for any $\mu>0$, we have
$$-\va^2\De w_\mu+\frac{1}{|x|^2}w_\mu\ge 0,\quad |x|\ge1$$
provided that $\va>0$ is small. However, for $\mu=N$ and any $\va>0$, there exists $R_\va>0$ satisfying
$$\va^{2s}(-\De)^s w_N+\frac{1}{|x|^{2s}}w_N\le 0,\quad |x|\ge R_\va.$$

Taking again the  iteration processes based on comparison principle and the estimates of $(-\De)^sw_\mu$, we also give the lower
decay estimates for any positive weak solutions of \eq{eqs1.1} in Theorem \r{thm1.5}.
Theorems \r{thm1.3} and \r{thm1.5} reveals some essential differences between the non-vanishing case $\inf_{\rn}V(x)>0$ and the vanishing case $V(x)\to 0$ at infinity.
In particular, if $V(x)\equiv\la>0$, as shown in \cite{paj2012}, \eq{eqs1.1} has a positive weak  solution for any $p\in(2,2_s^*)$, and moreover, any positive weak solution $u$ of \eq{eqs1.1} must satisfy the following upper decay
$$u(x)(1+|x|^{N+2s})<+\wq.$$
However, if $V(x)\le \frac{C}{1+|x|^{\om}}$ with $\om\in(0,2s)$, \eq{eqs1.1} has no positive weak solution for any $p\in(2,2+\frac{\om}{N+2s-\om})$, and besides, any possible positive weak solution $u$ of \eq{eqs1.1} must satisfy
$$u(x)(1+|x|^{N+2s})=+\wq.$$

\vspace{0.2cm}

We organize this paper as follows: In Section \ref{sec5}, we give some notations and known results on a penalized problem to \eq{eqs1.1}.
In Section \ref{SQ}, we recover the original problem and complete the proof of Theorem \r{thm1.2}.  In Section \r{s6}, we present some nonexistence results and then verify Theorem \r{thm1.3}. In Section \r{dey}, we derive some important decay estimates and complete the proof of Theorem \r{thm1.5} and Corollary \r{c30}.

\vspace{0.2cm}

\section{Preliminaries}\label{sec5}
In this section, we give some notations and preliminaries, and summarize some results on a penalized problem to \eq{eqs1.1} according to \cite{xsc2019}.

For $s\in(0,1)$, the fractional Sobolev space $H^s(\rn)$ is defined as
$$
H^s(\R^N)=\Big\{u\in L^2(\R^N):\iint_{\R^{2N}}\frac{|u(x)-u(y)|^2}{|x-y|^{N+2s}}<\infty\Big\}
$$
endowed with the  norm
$$\|u\|_{H^s(\rn)}=\Big(\|u\|_{L^2(\rn)}^2+\iint_{\R^{2N}}\frac{|u(x)-u(y)|^2}{|x-y|^{N+2s}}\Big)^{\frac{1}{2}}.$$
For $N>2s$, the space $\dot H^s(\R^N)$ is defined as
$$\dot H^s(\R^N)=\Big\{u\in L^{2_s^{\ast}}(\R^N):\iint_{\R^{2N}}\frac{|u(x)-u(y)|^2}{|x-y|^{N+2s}}<\infty\Big\},$$
endowed with the  norm
$$\|u\|_{\dot{H}^s(\rn)}=\Big(\iint_{\R^{2N}}\frac{|u(x)-u(y)|^2}{|x-y|^{N+2s}}\Big)^{\frac{1}{2}}.$$

%For $K\subset \R^N$, the local fractional Sobolev space is given by
%$$H^s(K)=\Big\{u\in L^2(K)\mid\iint_{K\times K}\frac{|u(x)-u(y)|^2}{|x-y|^{N+2s}}<\infty \Big\}.$$
In this paper, we  will use the following weighted Hilbert space
$$
H_{V, \varepsilon}^s(\R^N):=\Big\{u\in \dot{H}^s(\R^N):\int_{\R^N}V(x)u^2<\infty \Big\},$$
with the inner product
$$
\langle u,v\rangle_\varepsilon=\varepsilon^{2s}\iint_{\R^{2N}}\frac{\big(u(x)-u(y)\big)\big(v(x)-v(y)\big)}{|x-y|^{N+2s}}
+\int_{\R^N}
V(x)uv$$
and the corresponding norm
$$\|u\|_\varepsilon=\Big(\varepsilon^{2s}\|u\|^2_{\dot{H}^s(\rn)}+\int_{\R^N}V(x)u^2\Big)^{\frac{1}{2}}.$$

We give the following elementary embedding properties.
\begin{proposition}\label{prop2.1}{\rm(\cite{ege2012, lr2008} Embedding\ inequalities)}
 For any $u\in\dot{H}^s(\R^N)$, there exists  a constant $C>0$ depending only on $N$ and $s$ such that
$$
\int_{\R^N}\frac{|u(x)|^2}{|x|^{2s}}\le C\|u\|^2_{\dot{H}^s(\rn)},\quad \|u\|_{L^{2_s^*}(\rn)}\le C\|u\|_{\dot{H}^s(\rn)}.
$$
Moreover, the embedding $H^s(\R^N)\subset L^q(\R^N)$ is continuous for  $q\in [2, 2_s^*]$, and the embedding $\dot{H}^s(\rn)\subset L_{\mathrm{loc}}^q(\rn)$ for $q\in[1,2_s^{\ast})$ is compact.
\end{proposition}

Let $\{\mathcal{P}_\va\}_\va\subset L^\wq(\rn)$ be a family of nonnegative functions
satisfying
\begin{equation}\label{eqs2.1}
\mathcal{P}_{\varepsilon}(x)=0~ \text{for}~ x\in~\Lambda\ \text{and}\  \lim\limits_{\varepsilon\to 0} \| \mathcal{P}_{\varepsilon}\|_{L^\infty (\R^N)}=0.
\end{equation}
Moreover, we impose the following two embedding assumptions on $\{\mathcal{P}_\va\}_\va$,

$\left(\mathcal{P}_{1}\right)$  the space $H_{V,\varepsilon}^s\left(\mathbb{R}^{N}\right)$ is compactly embedded into $L^{2}\left(\mathbb{R}^{N}, \mathcal{P}_{\varepsilon}(x)\mathrm{d} x\right)$,

$\left(\mathcal{P}_{2}\right)$  there exists $\kappa\in (0,\frac{1}{2})$ such that
$$
 \int_{\mathbb{R}^{N}}\mathcal{P}_{\varepsilon}(x) u^2 \leq \kappa\Big( \va^{2s}\iint_{\R^{2N}}\frac{|u(x)-u(y)|^2}{|x-y|^{N+2s}}+\int_{\rn}V(x)u^2\Big).
$$
In this section, we do not need to know the specific form of $\mathcal{P}_\va$, the exact range of $p$ and the decay property of $V$, which only work in Section \r{SQ} to recover the original problem.

Consider the following auxiliary problem to \eq{eqs1.1}
\begin{equation}\label{qaj}
  \va^{2s}(-\De)^su+Vu=\chi_{\La}(x)u_+^{p-1}+\chi_{\rn\backslash\La}(x)g_\va(x,u)u_+,
\end{equation}
where $g_\va(x,t):=\min\{t_+^{p-2},\mathcal{P}_\va(x)\}$. Define $F_\va(x,t):=\int_0^tg_\va(x,r)r_+ dr$.
Then the functional corresponding to \eq{qaj} is given by
$$J_\va(u):=\frac{1}{2}\|u\|_\va^2-\frac{1}{p}\int_{\La}u_+^{p}-\int_{\rn\backslash\La}F_\va(x, u_+),\quad u\in H^s_{V,\va}(\rn).$$
By \cite[Lemma 2.5]{xsc2019}, $J_\va\in C^1$ in $H_{V,\va}^s(\rn)$ and satisfies the Mountain-Pass geometry. Define the Mountain-Pass value $c_{\va}$  as
\begin{equation}\label{Ade2.11}
  c_{\va}:=\inf_{\gamma\in \Gamma_{\va}}\max_{t\in[0,1]}J_{\varepsilon}(\gamma(t)),
\end{equation}
where
$\Gamma_{\va}:=\big\{\gamma\in C\big([0,1],\hv\big)\mid\gamma(0)=0,\ J_{\varepsilon}\big(\gamma(1)\big)<0\big\}.$

According to \cite[Lemmas 2.6 and 3.6]{xsc2019}, we have the following results on \eq{qaj}.
\begin{lemma}\label{lem5.1}(\cite{xsc2019})Let $V$ satisfy ($\mathcal{V}$),  $p\in (2,2_s^*)$ and $\left(\mathcal{P}_1\right)$-$\left(\mathcal{P}_2\right)$ hold.
Then $c_{\varepsilon}$ can be achieved by some $u_{\va}\in \hv\cap C(\rn)$, which is a positive weak solution of the penalized equation \eqref{qaj}.
Moreover,  there exists a family of points $\{x_\va\}_\va\subset \rn$ such that $u_\va(x_\va)=\max_{x\in \bar{\La}}u_\va(x)$ and

$({\rm\romannumeral1})~~\liminf\limits_{\va\to0}u_\va(x_\va)>0;$

$({\rm\romannumeral2})~\lim\limits_{\va\to0}V(x_{\varepsilon})=V_0;$

$({\rm\romannumeral3})\liminf\limits_{\va\to0}{\rm dist}(x_{\varepsilon}, \Lambda^c)>0;$

$({\rm\romannumeral4})\limsup\limits_{R\to\wq}\limsup \limits_{\va\to0}\|u_{\varepsilon}\|_{L^{\infty}(\La\setminus B_{\va R}(x_{\varepsilon}))}=0.$
\end{lemma}

As an important supplement, next we provide a uniform global $L^\wq$-estimate on $u_\va$ by Moser iteration, which will be used in Section \r{SQ}.
\begin{lemma}\label{w}
Let  $u_\va$ be given by Lemma \r{lem5.1}. Then there exists a constant $C>0$ independent of  $\va>0$ such that
\begin{equation*}\label{eqs4.6}
\|u_\va\|_{L^\wq(\rn)}\le C.
\end{equation*}
\end{lemma}
\begin{proof}
Note
$$\frac{1}{p}g_\va(x,t)t_+^2-F_\va(x,t)\ge \Big(\frac{1}{p}-\frac{1}{2}\Big)\mathcal{P}_\va t_+^2.$$
By ($\mathcal{P}_2$), we get
\begin{align}\label{q23}
J_\va(u_\va)-\frac{1}{p}\langle J'_\va(u_\va),u_\va\rangle=&\Big(\frac{1}{2}-\frac{1}{p}\Big)\|u_\va\|_\va^2+\int_{\rn\backslash\La}
\Big(\frac{1}{p}g_\va(x,u_\va)u^2_{\va,+}-F_\va(x,u_\va)\Big)\nonumber\\
&\ge\Big(\frac{1}{2}-\frac{1}{p}\Big)(1-\kappa)\|u_\va\|_\va^2.
\end{align}
By \cite[Lemma 3.2]{xsc2019}, $J_\va(u_\va)=c_\va\le C\va^N$. Then it follows from \eq{q23} and $J'_\va(u_\va)=0$ that
$\|u_\va\|^2_\va \le C\va^N.$
Define $\tilde{u}_\va(x)=u_\va(\va x)$. By Sobolev inequality and a change of variable, we obtain
\begin{equation}\label{dff}
\begin{aligned}
&\|\tilde{u}_\va\|^2_{L^{2_s^*}(\rn)}\le C\|\tilde{u}_\va\|^2_{\dot{H}^s(\rn)}= \frac{C}{\va^N}\va^{2s}\|u_\va\|^2_{\dot{H}^s(\rn)}
\le\frac{C}{\va^N}\|u_\va\|_\va^2\le C.
\end{aligned}
\end{equation}
Recalling that $u_\va$ is a positive weak solution to \eq{qaj}, by rescaling,  $\tilde{u}_\va$ weakly satisfies
\begin{equation}\label{d9}
\begin{aligned}
\Fs \tilde{u}_\va\le \tilde{u}_\va^{p-1},\quad x\in\rn.
\end{aligned}
\end{equation}

Let $\be\ge1$ and $T>0$. Define
\begin{equation*}\label{qdj}
\varphi_{\be,T}(t)=\left\{\begin{array}{l}
0, \text { if } t \leqslant 0, \vspace{1mm}\\
t^{\beta}, \text { if } 0<t<T, \vspace{1mm}\\
\beta T^{\beta-1}(t-T)+T^{\beta}, \text { if } t \geqslant T.
\end{array}\right.
\end{equation*}
Note that $\varphi_{\be,T}$ is convex and Lipschitz, there hold
\begin{align}\label{wql}
  \varphi_{\be,T}(\tilde{u}_\va), \varphi_{\be,T}'(\tilde{u}_\va)\ge0\ \mathrm{and}\ \varphi_{\be,T}(\tilde{u}_\va), \varphi_{\be,T}(\tilde{u}_\va)\varphi_{\be,T}'(\tilde{u}_\va)\in \dot{H}^s(\rn).
\end{align}
Moreover, since $\tilde{u}_\va>0$ and $\varphi_{\be,T}$ is convex, we have
\begin{align*}
 &\big|\varphi_{\be,T}(\tilde{u}_\va(x))-\varphi_{\be,T}(\tilde{u}_\va(y))\big|^2\\
 =&\big(\varphi_{\be,T}(\tilde{u}_\va(x))\big)^2+\big(\varphi_{\be,T}(\tilde{u}_\va(y))\big)^2
 -\varphi_{\be,T}(\tilde{u}_\va(y))\varphi_{\be,T}(\tilde{u}_\va(x))-\varphi_{\be,T}(\tilde{u}_\va(x))\varphi_{\be,T}(\tilde{u}_\va(y))\\
 \le&\big(\varphi_{\be,T}(\tilde{u}_\va(x))\big)^2+\big(\varphi_{\be,T}(\tilde{u}_\va(y))\big)^2
 -\varphi_{\be,T}(\tilde{u}_\va(y))\big(\varphi_{\be,T}(\tilde{u}_\va(y))+\varphi'_{\be,T}(\tilde{u}_\va(y))(\tilde{u}_\va(x)-\tilde{u}_\va(y))\big)\\
 &-\varphi_{\be,T}(\tilde{u}_\va(x))\big(\varphi_{\be,T}(\tilde{u}_\va(x))+\varphi'_{\be,T}(\tilde{u}_\va(x))(\tilde{u}_\va(y)-\tilde{u}_\va(x))\big)\\
 =&(\tilde{u}_\va(x)-\tilde{u}_\va(y))\big(\varphi_{\be,T}(\tilde{u}_\va(x))\varphi'_{\be,T}(\tilde{u}_\va(x))-
 \varphi_{\be,T}(\tilde{u}_\va(y))\varphi'_{\be,T}(\tilde{u}_\va(y))\big).
\end{align*}
It follows by Sobolev inequality, \eq{d9} and \eq{wql} that
\begin{align}\label{dk}
&\|\varphi_{\be,T}(\tilde{u}_\va)\|^2_{L^{2_s^*}(\rn)}\nonumber\\
\le& C\iint_{\R^{2N}}\frac{|\varphi_{\be,T}(\tilde{u}_\va(x))-\varphi_{\be,T}(\tilde{u}_\va(y))|^2}{|x-y|^{N+2s}}\nonumber\\
\le& C\iint_{\R^{2N}}\frac{(\tilde{u}_\va(x)-\tilde{u}_\va(y))\big(\varphi_{\be,T}(\tilde{u}_\va(x))\varphi_{\be,T}'(\tilde{u}_\va(x))
-\varphi_{\be,T}(\tilde{u}_\va(y))\varphi_{\be,T}'(\tilde{u}_\va(y))\big)}{|x-y|^{N+2s}}\nonumber\\
\le& C\int_{\rn}\tilde{u}_\va^{p-1}\varphi_{\be,T}(\tilde{u}_\va)\varphi_{\be,T}'(\tilde{u}_\va).
\end{align}
Since $t\varphi_{\be,T}'(t)\le\be\varphi_{\be,T}(t)$, we obtain from $\eq{dk}$ that
\begin{align}\label{dl}
&\|\varphi_{\be,T}(\tilde{u}_\va)\|^2_{L^{2_s^*}(\rn)}
\le C\be\int_{\rn}\tilde{u}_\va^{p-2}\big(\varphi_{\be,T}(\tilde{u}_\va)\big)^2.
\end{align}
Letting $T\to\wq$, by Monotone Convergence Theorem, we get
\begin{align}\label{ttog}
  \Big(\int_{\rn}\tilde{u}_\va^{\be2_s^*}\Big)^{\frac{2}{2_s^*}}\le C\be\int_{\rn}\tilde{u}_\va^{2\be+p-2}.
\end{align}
Set $\{\be_{i}\}_{i\ge1}$ so that
$$2\be_{i+1}+p-2=\be_i2_s^*,\ \ \beta_0=1,$$
i.e.,
$\be_{i+1}+d=\frac{2_s^*}{2}(\be_i+d)$ with $d=\frac{p-2}{2-2_s^*}> -1$.
Then, letting $\be=\be_{i+1}$ in $\eq{ttog}$, we derive
\begin{equation*}\label{eg}
\begin{aligned}
\Big(\int_{\rn}\tilde{u}_\va^{2_s^*\be_{i+1}}\Big)^\frac{1}{2_s^*(\be_{i+1}+d)}
\le(C\be_{i+1})^\frac{1}{2(\be_{i+1}+d)}\Big(\int_{\rn}\tilde{u}_\va^{2_s^*\be_{i}}\Big)^\frac{1}{2_s^*(\be_{i}+d)}.
\end{aligned}
\end{equation*}
As a result, it follows immediately by iteration and \eq{dff} that
\begin{align}
  \Big(\int_{\rn}\tilde{u}_\va^{2_s^*\be_{i}}\Big)^\frac{1}{2_s^*(\be_{i}+d)}\le\prod_{i=1}^{\infty}(C\be_i)^{\frac{1}{2(\be_i+d)}}
\Big(\int_{\rn}\tilde{u}_\va^{2_s^*}\Big)^\frac{1}{2_s^*(1+d)}\le C,\nonumber
\end{align}
where $C>0$ is a constant independent of $i$ and $\va$. Letting $i\to\wq$, we conclude that
$\|\tilde{u}_\va\|_{L^\infty(\rn)}\le C$
uniformly for $\va$. Then the proof is completed by the definition of $\tilde{u}_\va$.
\end{proof}

\vspace{0.2cm}

\section{Proof of Theorem \r{thm1.2}}\label{SQ}
In this section, we come back to the original problem \eq{eqs1.1} and complete the proof of Theorem \r{thm1.2}.
For convenience, in this section, we always
let $u_\va$ and $\{x_{\varepsilon}\}$ be given by Lemma \ref{lem5.1} and
$$v_{\varepsilon}(x)=u_{\varepsilon}(\varepsilon x + x_{\varepsilon}),\ V_\va(x):=V(\va x+x_\va),\ \tilde{\mathcal{P}}_{\varepsilon}(x)=\mathcal{P}_{\varepsilon}(\va x+x_{\varepsilon}).$$
By Lemma $\ref {lem5.1}$ (iii), we have
\begin{align}\label{ssa}
c_\La\va |x|\le|x_\va+\va x|\le C_\La\va|x|,\ \ x\in\rn\setminus\La_\va,
\end{align}
where $\La_\va:=\{x\mid x_\va+\va x\in\La\}$, $c_\La,C_\La>0$ are some constants independent of $x$ and $\va$.

We also define the set of test functions for the weak sub(super)-solutions outside a ball
\begin{align}
H_{c,R}^s(\rn):=\left\{\psi\in \Ds, \psi\ge 0\mid \mathrm{supp} \psi\ \mathrm{is\ compact},\ \psi=0\ \mathrm{in}\ B_R(0)\right\}.\nonumber
\end{align}

Next, we employ Lemma \r{lem5.1} (iv) to linearize the penalized problem.
\begin{proposition}\label{prop5.3}
Let $V$ satisfy ($\mathcal{V}$), $p\in (2,2_s^*)$ and  ($\mathcal{P}_1$)-($\mathcal{P}_2$) hold. Then there exist $R_0>0$ and $\va_{R}>0$ such that
for $R>R_0$ and $\va\in (0,\va_{R})$, $v_{\varepsilon}$ is a weak sub-solution to the following equation
\begin{equation}\label{q1t}
(-\Delta)^s v +\frac{1}{2}V_{\varepsilon}v= \tilde{\mathcal P}_\va v,\ x\in\R^N\setminus B_R(0),
\end{equation}
i.e.,
\begin{align*}
\iint_{\R^{2N}}\frac{(v_\va(x)-v_\va(y))(\var(x)-\var(y))}{|x-y|^{N+2s}}+\frac{1}{2}\int_{\rn}V_\va v_\va\var\le\int_{\rn}\tilde{\mathcal P}_\va v_\va\var,\quad \varphi\in H_{c,R}^s(\rn).
\end{align*}
\end{proposition}
\begin{proof}
By Lemma \r{lem5.1} (iv),  there exists $R_0>0$ and $\va_{R}>0$ such that
\begin{equation}\label{w20}
  u_\va^{p-2}\le \frac{1}{2} V_0\quad \mathrm{in}\ \La\backslash B_{\va R}(x_\va)
\end{equation}
for any $R>R_0$ and $0<\va<\va_{R}$.

Fix  $\varphi\in H_{c,R}^s(\rn)$. Denote $\varphi_{\varepsilon}(x):=\varphi (\frac{x - x_{\varepsilon}}{\varepsilon})$, then $\varphi_\va\in H_{V,\va}^s(\rn)$ and $\varphi_\va=0\ \mathrm{in}\ B_{\va R}(x_\va)$.
Recalling $g_\va(x,u_\va)\le\mathcal{P_\va}$, by \eq{w20} and $V_0\le V$ in $\La$, we have
\begin{equation}\label{w2o}
  \int_{\rn}(\chi_{\La}u_\va^{p-2}+\chi_{\La^c}g_\va(x,u_\va))u_\va\var_\va\le \frac{1}{2}\int_{\rn}V u_\va\var_\va+\int_{\rn}\mathcal{P_\va}u_\va\var_\va.
\end{equation}
Taking $\varphi_{\varepsilon}$ as a test function in  $\eq{qaj}$ for $u_\va$, it follows by \eq{w2o} that
\begin{align*}
 \va^{2s}\iint_{\R^{2N}}\frac{(u_\va(x)-u_\va(y))(\var_\va(x)-\var_\va(y))}{|x-y|^{N+2s}}+\frac{1}{2}\int_{\rn}V u_\va\var_\va\le
 \int_{\rn}\mathcal{P_\va}u_\va\var_\va.
\end{align*}
Then the conclusion holds immediately by rescaling.
\end{proof}

\begin{proposition}\label{5z}
Let $(\mathcal{P}_2)$ hold and $v\in \dot{H}^s(\rn)$ with
$\int_{\rn}V_\va v^2_+<\wq$.
If $v$ satisfies weakly
\begin{align}\label{n11}
(-\Delta)^s v +\frac{1}{2}V_{\varepsilon}v\leq \tilde{\mathcal P}_\va v\ \mathrm{in}\ \R^N\setminus B_R(0),
\end{align}
and $v\le0$ in $B_R(0)$, then $v\le0$ in $\rn$.
\end{proposition}
\begin{proof}
Clearly, $v_+=0$ in $B_R(0)$ and $v_+\in \dot{H}^s(\rn)$.
 Let $\var_n=\eta(\frac{x}{n})v_+$ where $\eta\in C_c^\wq(\rn,[0,1])$ satisfying
$\eta\equiv1$ in $B_R(0)$ and $\mathrm{supp}\eta\subset B_{2R}(0)$, then $\{\var_n\}_{n\ge1}\subset H_{c,R}^s(\rn)$. Moreover, by \cite[Lemma 5]{pap}, $\var_n\to v_+$ in $\dot{H}^s(\rn)$ as $n\to\wq$.

Taking $\var_n$ as a test function into $\eq{n11}$, we see that
\begin{align}\label{wbj}
&\iint_{\R^{2N}}\frac{(v_+(x)-v_+(y))(\var_n(x)-\var_n(y))}{|x-y|^{N+2s}}+\frac{1}{2}\int_{\rn}V_\va v_+\var_n\le \int_{\rn}\tilde{\mathcal P}_\va v_+\var_n,
\end{align}
where we use that $(-v_-(x)+v_-(y))(\var_n(x)-\var_n(y))\ge0$.

Since $\var_n\to v_+$ in $\dot{H}^s(\rn)$ as $n\to\wq$, it follows that
\begin{align}\label{24j}
  \lim_{n\to\wq}\iint_{\R^{2N}}\frac{(v_+(x)-v_+(y))(\var_n(x)-\var_n(y))}{|x-y|^{N+2s}}=
  \iint_{\R^{2N}}\frac{|v_+(x)-v_+(y)|^2}{|x-y|^{N+2s}}.
\end{align}
Clearly, since $\var_n\le v_+$, it holds $\int_{\rn}\tilde{\mathcal P}_\va v_+\var_n\le \int_{\rn}\tilde{\mathcal P}_\va v_+^2$.
Moreover, by Fatou's Lemma,
\begin{align}\label{25j}
 \int_{\rn}V_\va v_+^2\le\liminf_{n\to\wq} \int_{\rn}V_\va v_+\var_n.
\end{align}
Therefore,  letting $n\to\wq$ in \eq{wbj}, we get
\begin{align}\label{p02}
\iint_{\R^{2N}}\frac{|v_+(x)-v_+(y)|^2}{|x-y|^{N+2s}}+\frac{1}{2}\int_{\rn}V_\va v_+^2\le& \int_{\rn}\tilde{\mathcal P}_\va v_+^2.
\end{align}
Denote $\tilde{v}(x):=v_+(\frac{x-x_\va}{\va})$. Since $v_+\in \dot{H}^s(\rn)$ with $\int_{\rn}V_\va v_+^2<\wq$, we have $\tilde{v}\in H_{V,\va}^s(\rn)$. Then $\tilde{v}$ satisfies ($\mathcal{P}_2$).
It follows by rescaling that
\begin{align}\label{q2o}
   \int_{\mathbb{R}^{N}}\tilde{\mathcal{P}}_{\varepsilon}v_+^2 \leq \kappa\Big( \iint_{\R^{2N}}\frac{|v_+(x)-v_+(y)|^2}{|x-y|^{N+2s}}+\int_{\rn}V_\va v_+^2\Big).
\end{align}
Since $\kappa<\frac{1}{2}$, \eq{p02} and \eq{q2o} imply that $v_+=0$ in $\rn$, which completes the proof.
\end{proof}

Now we intend to find the super-solutions for the linear problem $\eq{q1t}$. Afterwards, we choose the sup-solutions of the form
\begin{align}\label{sn}
w_\mu=\frac{1}{(1+|x|^2)^\frac{\mu}{2}}
\end{align}
by adjusting the parameter $\mu>0$. Clearly, $w_\mu\in C^{\beta,\a}(\rn)$ for any $\beta>0$ and $\a\in(0,1)$. Specially, $\Fs w_\mu$ is well-defined pointwise.

 We  estimate  the nonlocal term $\Fs w_\mu$, whose proof is  technique and tedious. To keep the main clue of the paper clear, we
 postpone the proof  to Appendix \r{sok}.
\begin{proposition}\label{tb}
Let $\mu\in(0,+\infty)$. There exists constants $R_\mu, C_\mu, \tilde{C}_\mu>0$ depending  on $\mu$, $N$ and $s$ such that
\begin{align}
\left\{
  \begin{array}{ll}
    0<C_\mu\frac{1}{|x|^{\mu+2s}}\le\Fs w_\mu\le3C_\mu\frac{1}{|x|^{\mu+2s}}, & \mathrm{if }\ |x|>R_\mu\ \mathrm{and}\ \mu\in(0,N-2s); \vspace{1mm}\\
    \Fs w_\mu=C_{N-2s}w_\mu^{2_s^*-1},\ x\in \rn,& \mathrm{if }\ \mu=N-2s;\vspace{1mm}\\
-3C_\mu\frac{1}{|x|^{\mu+2s}}\le\Fs w_\mu\le -C_\mu\frac{1}{|x|^{\mu+2s}}<0,& \mathrm{if }\ |x|>R_\mu\ \mathrm{and}\ \mu\in(N-2s,N);\vspace{1mm}\\
   -\frac{\tilde{C}_{N}\ln|x|}{|x|^{N+2s}}\le\Fs w_\mu\le-\frac{C_{N}\ln|x|}{|x|^{N+2s}}<0,& \mathrm{if }\ |x|>R_\mu\ \mathrm{and}\ \mu=N,\vspace{1mm}\\
   -\frac{\tilde{C}_\mu}{|x|^{N+2s}}\le\Fs w_\mu\le -\frac{C_\mu}{|x|^{N+2s}}<0, & \mathrm{if }\ |x|>R_\mu\ \mathrm{and}\ \mu>N.
\end{array}
\right.\nonumber
\end{align}
\end{proposition}

Moreover, we give some important integral properties of $w_\mu$.
\begin{proposition}\label{cc5} There hold:

(i) $w_\mu\in \dot{H}^s(\rn)$ for $\mu>\frac{N-2s}{2}$ and $w_\mu\notin \dot{H}^s(\rn)$ for $0<\mu\le\frac{N-2s}{2}$;

(ii) For any $\mu>0$ and $\phi\in C_c^\wq(\rn)$,
\begin{align*}%\label{qr9}
  \iint_{\R^{2N}}\frac{(w_\mu(x)-w_\mu(y))(\phi(x)-\phi(y))}{|x-y|^{N+2s}}=\int_{\rn}(-\De)^{s}w_\mu\phi;
  \end{align*}

(iii) For any $\mu>\frac{N-2s}{2}$ and $\phi\in \dot{H}^s(\rn)$,
\begin{align*}%\label{qr9}
  \iint_{\R^{2N}}\frac{(w_\mu(x)-w_\mu(y))(\phi(x)-\phi(y))}{|x-y|^{N+2s}}=\int_{\rn}(-\De)^{s}w_\mu\phi.
\end{align*}
\end{proposition}
\begin{proof}
We first consider (i). Clearly, $\int_{\rn}w_\mu^{2_s^*}=+\wq$ for $\mu\le\frac{N-2s}{2}$, and $\int_{\rn}w_\mu^{2_s^*}<\wq$ for $\mu>\frac{N-2s}{2}$. On the other hand, by symmetry,
\begin{align}\label{agj}
  \int_{\rn}\int_{\rn}\frac{|w_\mu(x)-w_\mu(y)|^2}{|x-y|^{N+2s}}dxdy
  =&2\int\int_{\{|x|>|y|\}}\frac{|w_\mu(x)-w_\mu(y)|^2}{|x-y|^{N+2s}}dxdy\nonumber\\
  =&2\int\int_{\{|x|>|y|\}\cap\{|x-y|\le|x|/2\}}\frac{|w_\mu(x)-w_\mu(y)|^2}{|x-y|^{N+2s}}dxdy\nonumber\\
  &+2\int\int_{\{|x|>|y|\}\cap\{|x-y|>|x|/2\}}\frac{|w_\mu(x)-w_\mu(y)|^2}{|x-y|^{N+2s}}dxdy\nonumber\\
  :=&2I_1+2I_2.
\end{align}
Note that $|\nabla w_\mu(z)|\le C|\nabla w_\mu(x)|$ for all
$\frac{|x|}{2}\le|z|\le|x|$. Then by mean value theorem,
\begin{align*}
  I_1\le& C\int\int_{\{|x|>|y|\}\cap\{|x-y|\le|x|/2\}}\frac{|\nabla w_\mu(x)|^2}{|x-y|^{N+2s-2}}dxdy\\
  \le&C\int_{\rn}|\nabla w_\mu(x)|^2\Big(\int_{\{|x-y|\le|x|/2\}}\frac{1}{|x-y|^{N+2s-2}}dy\Big)dx\\
  =&C\int_{\rn}|\nabla w_\mu(x)|^2|x|^{2-2s}dx= C\mu^2\int_{\rn}\frac{|x|^{4-2s}}{(1+|x|^2)^{\mu+2}}dx<\wq
\end{align*}
provided that $\mu>\frac{N-2s}{2}$. Clearly, $|w_\mu(x)-w_\mu(y)|^2< |w_\mu(y)|^2$ for $|x|>|y|$, then
\begin{align}\label{sdf}
 I_2\le&\int\int_{\{|x|>|y|\}\cap\{|x-y|>|y|/2\}}\frac{|w_\mu(y)|^2}{|x-y|^{N+2s}}dxdy\nonumber\\
 \le&\int_{\rn}|w_\mu(y)|^2\left(\int_{\{|x-y|>|y|/2\}}\frac{1}{|x-y|^{N+2s}}dx\right)dy\nonumber\\
 =&C\int_{\rn}(1+|y|^2)^{-\mu}|y|^{-2s}dy<\wq
\end{align}
provided $\mu>\frac{N-2s}{2}$. Thus (i) is proved.

Next we consider (ii).
Let $\mu>0$. For any $x\in\rn$,  by symmetry, we get
\begin{align*}
\Fs w_\mu=&2\int_{\rn}\frac{w_\mu(x)-w_\mu(y)+\chi_{B_1(x)}(y)\nabla w_\mu(x)\cdot(x-y)}{|x-y|^{N+2s}}d y.
\end{align*}
Since $w_\mu\le1$ and $D^2w_\mu\in L^\wq(\rn)$, by Taylor expansion, we get
\begin{align}\label{e4o}
&\int_{\rn}\frac{|w_\mu(x)-w_\mu(y)+\chi_{B_1(x)}(y)\nabla w_\mu(x)\cdot(x-y)|}{|x-y|^{N+2s}}d y\nonumber\\
  \le&\int_{\rn\setminus B_1(x)}\frac{2}{|x-y|^{N+2s}}d y+\int_{ B_1(x)}\frac{C|x-y|^2}{|x-y|^{N+2s}}d y\le C,
\end{align}
which also implies $\Fs w_\mu\in L^\wq(\rn)$.

Fix any $\varphi\in C_c^\wq(\rn)$ and any $r\in(0,1)$. By symmetry and Fubini Theorem, we have
\begin{align*}
 &\int_{\rn}\Big(2\int_{\mathbb{R}^N\backslash B_r(x)}\frac{w_\mu(x)-w_\mu(y)+\chi_{B_1(x)}(y)\nabla w_\mu(x)\cdot(x-y)}{|x-y|^{N+2s}}~\d y\Big)\varphi(x)\d x\\
  =&\int_{\rn}\Big(2\int_{\mathbb{R}^N\backslash B_r(x)}\frac{w_\mu(x)-w_\mu(y)}{|x-y|^{N+2s}}~\d y\Big)\varphi(x)\d x\\
  =&\int_{\rn}\Big(\int_{\mathbb{R}^N\backslash B_r(x)}\frac{w_\mu(x)-w_\mu(y)}{|x-y|^{N+2s}}~\d y\Big)\varphi(x)\d x+\int_{\rn}\Big(\int_{\mathbb{R}^N\backslash B_r(y)}\frac{w_\mu(y)-w_\mu(x)}{|y-x|^{N+2s}}~\d x\Big)\varphi(y)\d y\\
  =&\iint_{|x-y|\ge r}\frac{\big(w_\mu(x)-w_\mu(y)\big)\big(\varphi(x)-\varphi(y)\big)}{|x-y|^{N+2s}}\d x\d y.
\end{align*}
Letting $r\to0$, by \eq{e4o} and Dominated convergence theorem, we get
\begin{align*}%\label{qr9}
 \int_{\rn}(-\De)^{s}w_\mu\var=\iint_{\R^{2N}}\frac{(w_\mu(x)-w_\mu(y))(\var(x)-\var(y))}{|x-y|^{N+2s}},
  \end{align*}
which implies (ii) holds.

Finally, we prove (iii). Letting $\mu>\frac{N-2s}{2}$, we have $w_\mu\in \dot{H}^s(\rn)$ by (i), and $(-\De)^sw_\mu\in L^{\frac{2N}{N+2s}}(\rn)\subset (\dot{H}^s(\rn))^{-1}$ by Proposition \r{tb}. As a result, (iii) holds immediately by (ii) and
a density argument since $C_c^\wq(\rn)$ is dense in $\dot{H}^s(\rn)$.
\end{proof}

From now on,  we  assume the prescribed form of the penalization:
\begin{align}\label{ph}
\P(x)=\frac{\va^\theta}{|x|^\tau}\chi_{\La^c},
\end{align}
where $\theta, \tau>0$ are two parameters which will be determined later.

\begin{proposition}\label{ed5}(Construction of super-solutions)
Assume that one of the following four conditions holds:

($\mathcal{A}_1$) $\theta>\tau$, $\tau>2s$ and $\mu\in (0, N-2s)$;

($\mathcal{A}_2$) $\theta-\tau>-2s$,  $\tau> 2s$ and $\mu\in (N-2s,N)$ when $V$ satisfies $(\mathcal{V}_{lsd})$ with $\om=2s$;

($\mathcal{A}_3$) $\theta-\tau>-\om$,  $\tau> \om$ and $\mu\in (N, N+2s-\om)$ when $V$ satisfies $(\mathcal{V}_{lsd})$ with $\om\in[0,2s)$;

($\mathcal{A}_4$) $\theta>\tau> 0$ and $\mu\in (N, N+2s)$ when $V$ satisfies $(\mathcal{V}_{log})$.
%Assume that $V$ satisfies $\inf_{\rn}V(1+|x|^{2s})\ge0$ \textrm{or} $\inf_{\rn}V(1+|x|^{\omega})>0$ for some $\omega\in(0,2s]$.

\noindent Then $w_\mu$ is a super-solution of (\ref {q1t}) in the classical sense, i.e.
\begin{align}\label{wdd}
\Fs w_\mu+\frac{1}{2}V_\va w_\mu\ge\tilde{\mathcal P}_\va w_\mu,\ x\in\R^N\setminus B_R(0)
\end{align}
for given $R>0$ large enough and $\va>0$ small enough.
\end{proposition}
\begin{proof}
For given large $R>\{R_\mu,1\}$, by Lemma \r{lem5.1} (iii), $B_R(0)\subset \La_\va$ for small $\va$. Recalling \eq{ssa}, we deduce that the right hand side of $\eq{wdd}$ satisties
\begin{align}\label{qj7}
 \tilde{\mathcal P}_\va w_\mu\le \frac{\va^{\theta-\tau}}{|x|^{\tau+\mu}}\chi_{\La_\va^c}.
\end{align}
Now we are going to estimate  the left hand side of $\eq{wdd}$  under the  various assumptions   ($\mathcal{A}_1$)  -($\mathcal{A}_4$) .

\vspace{0.1cm}

\textbf{The case under the assumption ($\mathcal{A}_1$):} For $\mu\in (0, N-2s)$,
from Proposition \r{tb}, we have
 $$\Fs w_\mu  +\frac{1}{2}V_{\va} w_\mu  \ge\frac{C}{|x|^{\mu+2s}} \  for  \ |x|>R.$$

\vspace{0.1cm}

\textbf{The case under the assumption ($\mathcal{A}_2$):}
From Proposition \r{tb}, for $R$ large, we have
$$\Fs w_\mu+\frac{1}{2}V_{\va} w_\mu\ge -\frac{3C_\mu}{|x|^{2s}}\frac{1}{|x|^\mu}+\frac{1}{2}V_0\frac{1}{(1+|x|^2)^{\mu/2}}\ge0,\ \ x\in \La_\va\setminus B_R(0).$$
Since $\inf_{x\in\rn}V(1+|x|^{2s})>0$, there exists $C>0$ such that $V(x)\ge \frac{C}{|x|^{2s}}$ for $|x|\ge1$.
By $\eq{ssa}$,  for $\va>0$ small,
 we have
$$\Fs w_\mu+\frac{1}{2}V_{\va} w_\mu\ge-\frac{3C_\mu}{|x|^{2s+\mu}}+\frac{C}{2C_\La^{2s}\va^{2s}}
\frac{1}{|x|^{2s}}\frac{1}{(1+|x|^2)^{\mu/2}}\ge\frac{C\va^{-2s}}{|x|^{\mu+2s}},\quad x\in\R^N\backslash\Lambda_\va.$$

\textbf{The case under the assumption ($\mathcal{A}_3$):}
%$\inf_{x\in\rn}V(x)(1+|x|^{\omega})>0$ for some $\omega\in [0,2s)$ and $N<\mu< N+2s-\omega$.
From Proposition \r{tb}, we get for $R$ large and $\va$ small that
$$\Fs w_\mu+\frac{1}{2}V_{\va} w_\mu\ge -\frac{\tilde{C}_\mu}{|x|^{N+2s}}+\frac{1}{2}V_0\frac{1}{(1+|x|^2)^{\mu/2}}\ge0,\quad x\in\La_\va\setminus B_R(0).$$
Since $\inf_{x\in\rn}V(1+|x|^{\om})>0$ for some $\omega\in [0,2s)$, there exists $C_\om>0$ such that $V(x)\ge \frac{C_\om}{|x|^{\om}}$ for $|x|\ge1$. Thus for $\va>0$ small,
it follows from $\eq{ssa}$, Proposition \r{tb} and $\mu+\om<N+2s$ that
$$\Fs w_\mu+\frac{1}{2}V_{\va} w_\mu\ge-
\frac{\tilde{C}_\mu}{|x|^{N+2s}}+\frac{C_\omega}{2C_\La^{\omega}\va^{\omega}}\frac{1}{|x|^{\om}}\frac{1}{(1+|x|^2)^{\mu/2}}\ge\frac{C\va^{-\omega}}{|x|^{\mu+\omega}}$$
for all $x\in\rn\setminus \La_\va$.

\textbf{The case under the assumption ($\mathcal{A}_4$):}
%$\inf_{x\in\rn}V(x)\log(e+|x|^2)>0$ and $N<\mu< N+2s$.
From Proposition \r{tb}, we get for $R$ large and $\va$ small that
$$\Fs w_\mu+\frac{1}{2}V_{\va} w_\mu\ge -\frac{\tilde{C}_\mu}{|x|^{N+2s}}+\frac{1}{2}V_0\frac{1}{(1+|x|^2)^{\mu/2}}\ge0,\quad x\in\La_\va\setminus B_R(0).$$
Since $\inf_{x\in\rn}V(x)\log(e+|x|^2)>0$, $V(x)\ge \frac{C}{\log(e+|x|^2)}$ for some $C>0$. For $\va$ small enough, it follows from $\eq{ssa}$ that
$$V_\va(x)=V(\va x+x_\va)\ge \frac{C}{\log(e+|\va x+x_\va|^2)}\ge \frac{C}{\log(e+C_\La^2\va^2|x|^2)}\ge \frac{C}{\log(e+|x|^2)}.$$
Therefore, for $\va>0$ small,
it follows from $\eq{ssa}$, Proposition \r{tb} and $\mu<N+2s$ that
$$\Fs w_\mu+\frac{1}{2}V_{\va} w_\mu\ge-
\frac{\tilde{C}_\mu}{|x|^{N+2s}}+\frac{C}{\log(e+|x|^2)}\frac{1}{(1+|x|^2)^{\mu/2}}\ge\frac{C}{|x|^{\mu}\log(e+|x|^2)}$$
for  $x\in\rn\setminus \La_\va$.

The conclusion follows immediately by summarizing above four cases with (\ref {qj7}) and the corresponding assumptions on $\tau$, $\theta$ and $\omega$.
\end{proof}

Next we are on the point of checking the two pre-assumptions ($\mathcal{P}_1$)-($\mathcal{P}_2$) in Section 2.
\begin{proposition}\label{w34}
Assume that one of the following three conditions holds:

 ($\mathcal{B}_1$) $\theta>2s$ and $\tau>2s$;

 ($\mathcal{B}_2$) $\theta>0$ and $\tau>\om$ when
$V$ satisfies $(\mathcal{V}_{lsd})$;

($\mathcal{B}_3$) $\theta>0$ and $\tau>0$ when
$V$ satisfies $(\mathcal{V}_{log})$.

\noindent Then
the penalized function $\mathcal{P}_{\varepsilon}$ given by (\ref {ph})
satisfies ($\mathcal{P}_1$) and ($\mathcal{P}_2$).
\end{proposition}
\begin{proof}
We only verify ($\mathcal{P}_1$) and ($\mathcal{P}_2$) under the assumptions ($\mathcal{B}_1$) and ($\mathcal{B}_2$). The case under the assumption ($\mathcal{B}_3$) is very similar to the case of ($\mathcal{B}_2$) and we omit the details here.

We first verify ($\mathcal{P}_2$).

 {\bf The case under the assumption ($\mathcal{B}_1$):} For any $\var\in \Ds$, by the assumption $\theta>2s$, $\tau>2s$ and Hardy inequality (see Proposition \r{prop2.1}),  we have,
\begin{align*}
 \int_{\mathbb{R}^{N}}\mathcal{P}_{\varepsilon}(x) \var^2=\va^{\theta}\int_{\La^c}\frac{1}{|x|^{\tau}}\var^2
\le C\va^{\theta}\int_{\La^c}\frac{|\var|^2}{|x|^{2s}}\le\k\va^{2s}\|\var\|^2_{\dot{H}^s(\rn)} \ for \
small \ \va>0,
\end{align*}
which implies $(\mathcal{P}_2)$.

{\bf The case under the assumption ($\mathcal{B}_2$):} Clearly, there exists a $C_\omega>0$ such that $V\ge\frac{C_\omega}{|x|^\omega}$ in $\rn\setminus\La$. By the assumptions $\theta>0$ and $\tau>\om$, for $\va$ small,  we have
\begin{align*}
 \int_{\mathbb{R}^{N}}\mathcal{P}_{\varepsilon}(x) \var^2=\va^{\theta}\int_{\La^c}\frac{1}{|x|^{\tau}}\var^2
\le\k C_\om\int_{\La^c}\frac{1}{|x|^\omega}\var^2\le\k\int_{\rn} V\var^2,
\end{align*}
which also implies $(\mathcal{P}_2)$.

 Next we turn to check ($\mathcal{P}_1$). Let $\{v_n\}_{n\in\N}$ is a bounded sequence in $\hv$. Up to a subsequence, there exists
some $v\in\hv$ such that $v_n\rightharpoonup v$ in $\hv$ and $v_n\to v$ in $L_{\mathrm{loc}}^q(\rn)$ for $q\in [1,2_s^*)$. Let $\va<1$ and $M>1$ such that
$\La\subset B_M(0)$.

{\bf The case under the assumption ($\mathcal{B}_1$):} By the assumption $\theta>2s$, $\tau>2s$ and Hardy inequality,
\begin{align}\label{zr}
\int_{\rn}|v_n-v|^2\P\nonumber
=&\va^{\theta}\int_{\rn\setminus B_M(0)}|v_n-v|^2\frac{1}{|x|^{\tau}}+\va^{\theta}\int_{B_M(0)\setminus\La}|v_n-v|^2\frac{1}{|x|^{\tau}}\nonumber\\
\le&\frac{\va^{2s}}{M^{\tau-2s}}\int_{\rn\setminus B_M(0)}\frac{|v_n-v|^2}{|x|^{2s}}+C\int_{B_M(0)\setminus\La}|v_n-v|^2,\\
\le&\frac{C}{M^{\tau-2s}}(\sup_{n\in\N}\va^{2s}\|v_n\|^2_{\dot{H}^s(\rn)}+\va^{2s}\|v\|^2_{\dot{H}^s(\rn)})
+C\int_{B_M(0)\setminus\La}|v_n-v|^2,\nonumber
\end{align}
which implies $v_n\to v$ in $L^2\big(\rn, \P\d x\big)$ as $n\to\wq$ and thereby ($\mathcal{P}_1$) holds.

{\bf The case under the assumption ($\mathcal{B}_2$):} Note $V\ge\frac{C_\omega}{|x|^\omega}$ in $\rn\setminus\La$. By the assumption $\theta>0$ and $\tau>\om$,
\begin{align}\label{zr}
\int_{\rn}|v_n-v|^2\P\nonumber
=&\va^{\theta}\int_{\rn\setminus B_M(0)}|v_n-v|^2\frac{1}{|x|^{\tau}}+\va^{\theta}\int_{B_M(0)\setminus\La}|v_n-v|^2\frac{1}{|x|^{\tau}}\nonumber\\
\le&\frac{1}{M^{\tau-\omega}}\int_{\rn\setminus B_M(0)}\frac{|v_n-v|^2}{|x|^{\omega}}+C\int_{B_M(0)\setminus\La}|v_n-v|^2\\
\le&\frac{C}{M^{\tau-\omega}}(\sup_{n\in\N}\int_{\rn}Vv_n^2+\int_{\rn}Vv^2)+C\int_{B_M(0)\setminus\La}|v_n-v|^2,\nonumber
\end{align}
which indicates $v_n\to v$ in $L^2\big(\rn, \P\d x\big)$ as $n\to\wq$ and so ($\mathcal{P}_1$) holds.

Hence  the proof is completed.
\end{proof}

Now we apply the comparison principle in Proposition \r{5z} to get the decay estimates of $u_\va$.
\begin{proposition}\label{wq8}
Let $N>2s$, $p\in (2,2_s^*)$ and $V$ satisfy $(\mathcal{V})$. If one of  ($\mathcal{A}_1$)-($\mathcal{A}_4$) holds,
then ($\mathcal{P}_1$)-($\mathcal{P}_2$) hold and there exists $C>0$ independent of small $\va$  such that
$v_\va\le Cw_\mu$. In  particular,
\begin{align}\label{wsq}
u_\va\le\frac{C\va^\mu}{|x|^\mu}\ \mathrm{in}\ \rn\setminus\La.
\end{align}
%where  $u_\va$ is given by Lemma \r{lem3.6} and $\{x_{\varepsilon}\}_\va$ is given by Lemma \ref{jz}.
\end{proposition}
\begin{proof}
Clearly, ($\mathcal{P}_1$)-($\mathcal{P}_2$) hold by Proposition \r{w34}, and \eq{wdd} holds by Proposition \r{ed5}.

By Lemma \r{w}, we have $A_0:=\sup_{\va\in(0,\va_0)}\|v_\va\|_{L^\wq(\rn)}<\wq$.
Fix $R$ large enough and let $\bar{w}_\mu=A_0(1+R^{\mu/2})^2 w_\mu$, $\bar{v}_\va=v_\va-\bar{w}_\mu$. Clearly, $\bar{v}_\va\le0$ in $B_R(0)$, $\bar{v}_\va\in \dot{H}^s(\rn)$ and $\int_{\rn}V_\va(\bar{v}_{\va,+})^2\le\int_{\rn}V_\va v^2_\va<\wq$. Moreover, from Proposition \r{prop5.3}, \eq{wdd} and Proposition \r{cc5} (iii), $\bar{v}_\va$  satisfies weakly
$$\Fs\bar{v}_\va+\frac{1}{2}V_{\va}\bar{v}_\va\le \tilde{\mathcal P}_\va\bar{v}_\va\ \ \mathrm{in}\ \R^N\backslash B_R(0).$$
It follows  from Proposition \r{5z} that
$\bar{v}_\va\le0$ in $\rn$. Then $v_\va\le Cw_\mu$. In particular, if $x\in \rn\setminus\La$, noting that $\liminf_{\va\to0}\dist(x_\va,\rn\setminus\La)>0$, it holds
\begin{align}\label{sjj}
u_\va(x)=v_\va\Big(\frac{x-x_\va}{\va}\Big)\le\frac{C}{\left(1+|\frac{x-x_\va}{\va}|^2\right)^{\frac{\mu}{2}}}\le \frac{C\va^\mu}{\va^\mu+|x-x_\va|^\mu}\le\frac{C\va^\mu}{|x|^\mu}.
\end{align}
This completes the proof.
\end{proof}

\noindent\textbf{Proof of Theorem \r{thm1.2}}:

{\bf  The case under the assumption ($\mathcal{Q}_2$) with $\om=2s$, i.e. $p>2+\frac{2s}{N}$, $\inf_{x\in\rn}V(x)(1+|x|^{2s})>0$.}

In this case, we select $\tau$ and $\theta$, and $\mu\in (N-2s,N)$  sufficiently close to $N$, such that
 \begin{align}\label{lll}
 2s<\tau<\theta<\mu(p-2)<N(p-2).
 \end{align}
Recall that  $\mathcal{P}_\va=\frac{\va^\theta}{|x|^\tau}\chi_{\La^c}$  and $w_\mu=\frac{1}{(1+|x|^2)^\frac{\mu}{2}}$. It follows by Proposition \r{wq8} that
\begin{align*}
 u_\va^{p-2}\le\frac{C\va^{\mu(p-2)}}{|x|^{\mu(p-2)}}<\mathcal{P}_\va\ \mathrm{in}\ \rn\setminus\La
\end{align*}
provided $\va$ small enough. Hence $u_\va$ is indeed a solution to the original problem \eq{eqs1.1}. Moreover, from $u(x_\va+\va x)\le Cw_\mu$ and $\liminf_{\va\to0}\dist(x_\va,\La^c)>0$, we get
$$u_\va\le \frac{C\va^\mu}{\va^\mu+|x-x_\va|^\mu},\quad x\in\rn.$$
By Lemma \r{w}, $u_\va\in L^\wq(\rn)$. Then $-Vu_\va+u_\va^{p-1}\in L^\wq_{\mathrm{loc}}(\rn)$, it follows from \c[Proposition 5]{re} that
$u_\va\in C^{\si}_{\mathrm{loc}}(\rn)$ for any $\si\in(0,\min\{2s,1\})$.
This completes the proof.

The proofs for the other cases are similar. So we only  provide  the selection of $p$ and parameters and  omit the detailed calculations here.

{\bf  The case under the assumption ($\mathcal{Q}_2$) with $\om\in[0,2s)$, i.e. $p>2+\frac{\om}{N+2s-\om}$, $\inf_{x\in\rn}V(x)(1+|x|^{\om})>0$.}

In this case, we select $\tau$ and $\theta$, and $\mu\in (N,N+2s-\om)$ sufficiently close to $N+2s-\om$,  such that
 \begin{align*}
 \om<\tau<\theta<\mu(p-2)<(N+2s-\om)(p-2).
\end{align*}

{\bf  The case under the assumption ($\mathcal{Q}_1$), i.e. $p>2+\frac{2s}{N-2s}$.}

In this case, we select  $\tau$ and $\theta$, and $\mu\in (0,N-2s)$ sufficiently close to $N-2s$, such that
 \begin{align*}
 2s<\tau<\theta<\mu(p-2)<(N-2s)(p-2).
\end{align*}

{\bf  The case under the assumption ($\mathcal{Q}_3$), i.e. $p>2$, $\inf_{x\in\rn}V(x)\log(e+|x|^2)>0$.}

In this case, we select $\tau$ and $\theta$, and  $\mu\in (N,N+2s)$   sufficiently close to $N+2s$,  such that
 \begin{align*}
 0<\tau<\theta<\mu(p-2)<(N+2s)(p-2).
\end{align*}

As a result, we complete the proof of Theorem \r{thm1.2}.

\vspace{0.2cm}

\section{Nonexistence results}\label{s6}
In this section, we aim to obtain some nonexistence results for \eq{eqs1.1}. Before that, we present the following comparison principle for fractional Laplacian.
\begin{lemma}\label{ws8}(Comparison principle) Let  $f(x)\in L^1_{\mathrm{loc}}(\rn\backslash\{0\})$ with $f(x)\ge0$. Suppose that $\tilde v\in \dot{H}(\rn)\cap C(\rn)$ with $\tilde v>0$  is  a weak supersolution to
\begin{align*}
  (-\De)^s v+Vv= f(x),\quad x\in \rn\backslash B_{R}(0),
\end{align*}
and that $\underline{v}_\lambda \in \dot{H}(\rn)\cap C(\rn)$ with $\underline{v}_\lambda >0$ is  a weak subsolution to
\begin{align*}
  (-\De)^s {v}+V{v}= \la f(x),\quad x\in \rn\backslash B_{R'}(0),
\end{align*}
where $R,R',\la>0$ are constants. Then there holds
\begin{align*}
  \tilde v\ge C \underline {v}_\lambda,\quad x\in\rn,
\end{align*}
where $C>0$ is a constant depending  on $\la$, $\tilde{R}:=\max\{R,R'\}$,
$\min_{B_{\tilde{R}}(0)} \tilde v$ and $\max_{B_{\tilde{R}}(0)}\underline{v}_\lambda$.
\end{lemma}

\begin{proof}
Define $\tilde{R}:=\max\{R,R'\}$, $\bar{v}:=\min\{1,\frac{\min_{B_{\tilde{R}}(0)}\tilde v}{\max_{B_{\tilde{R}}(0)}\underline{v}_\lambda}\}\frac{1}{\max\{1,\la\}}\underline{v}_\lambda$
and $w:=\bar{v}-\tilde v$. Clearly, $w\le0$ in $B_{\tilde{R}}(0)$ and $w$ weakly satisfies
\begin{align}\label{wq7}
  (-\De)^s w+Vw\le 0,\quad x\in \rn\backslash B_{\tilde{R}}(0).
\end{align}
Then by the same arguments as \eq{p02}, we get $w_+\le0$ in $\rn$, which  completes the proof.
\end{proof}

%Moreover, we give a general strong maximum principle, which will be used in the proof of Theorem \r{thm1.3}.

Now we are going to prove Theorem \r{thm1.3}.
 Without loss of  generality, we  fix $\va =1$. It suffices to consider the following equation
\begin{align}\label{wqj}
  (-\De)^su+V(x)u=u^{p-1},\,\,\, u>0, \quad x\in\rn.
\end{align}
%Denote
%$$H_V(\rn):=H_{V,1}^s(\rn)=\Big\{u\in \dot{H}^s(\rn)\Big| \int_{\rn}Vu^2<\wq\Big\}.$$

\begin{proof}[{\bf Proof of Theorem \r{thm1.3}}]
Suppose by contradiction that $u\in H^s_{V,1}(\rn)$ is a positive weak solution to \eq{wqj}.
By \cite[Proposition 4.1.1.]{dmv}, we have $u\in L^\wq(\rn)$. It follows by \cite[Proposition 2.9]{sil} that $u\in C(\rn)$.  Since $u$ is a weak solution to \eq{wqj}, we have
\begin{align}\label{sep}
  \int_{\rn}u^{p}=\|u\|^2_{\dot{H}^s(\rn)}+\int_{\rn}Vu^2<\wq.
\end{align}

Now we prove our nonexistence results under the  various assumptions on $V$ and $p$.

\noindent {\bf  The case under the assumption ($\mathcal{Q}'_1$)}, i.e.
$2<p<2+\frac{2s}{N-2s}$ and $\ds \lim_{|x|\to\wq}(1+|x|^{2s})V(x)=0$.

The assumption $\lim_{|x|\to\wq}(1+|x|^{2s})V(x)=0$ implies that for any $\epsilon>0$, there exists $R_\epsilon>0$ such that
$$V(x)\le \frac{\epsilon}{1+|x|^{2s}},\quad x\in \,\rn\backslash B_{R_\epsilon}(0).$$
 In the following, we can take $\epsilon>0$ small enough and $R>R_\epsilon$ if necessary.

Let $\mu_1\in (N-2s,N)$ be a parameter which will be determined later. By Proposition \r{tb} and Proposition \r{cc5} (iii), $w_{\mu_1}$ weakly satisfies
\begin{align}\label{qy9}
  (-\De)^s w_{\mu_1}+V(x)w_{\mu_1}\le-\frac{C_{\mu_1}}{|x|^{\mu_1+2s}}+\frac{\epsilon}{|x|^{\mu_1+2s}}\le0,\quad x\in \rn\backslash B_{R_1}(0)
\end{align}
for some $R_1>0$.
Then by Lemma \r{ws8}, there exists $C_1>0$ such that
\begin{align*}
  u\ge C_1w_{\mu_1}.
\end{align*}
Recalling \eq{wqj}, we obtain
\begin{align*}
  (-\De)^s u+V(x)u\ge \frac{C}{|x|^{\mu_1(p-1)}},\quad x\in \rn\backslash B_1(0).
\end{align*}
Since $p<2+\frac{2s}{N-2s}$, we can choose  $\mu_2\in (\frac{N-2s}{2},N-2s)$ and $\mu_1$ close to $N-2s$ such that
\begin{equation}\label{w60}
  N>\mu_2+2s>\mu_1(p-1).
\end{equation}
We deduce by Proposition \r{tb} and Proposition \r{cc5} (iii) that $w_{\mu_2}$ weakly satisfies
\begin{align*}
  (-\De)^s w_{\mu_2}+V(x)w_{\mu_2}\le \frac{C_{\mu_2}}{|x|^{\mu_2+2s}}+\frac{\epsilon}{|x|^{\mu_2+2s}}\le \frac{C}{|x|^{\mu_1(p-1)}},\quad x\in \rn\backslash B_{R_2}(0)
\end{align*}
for some $R_2>0$.
Consequently,  by Lemma \r{ws8}, there exists $C_2>0$ such that
\begin{align*}
  u\ge C_2w_{\mu_2}.
\end{align*}
Reviewing \eq{wqj}, we get
\begin{align*}
  (-\De)^s u+V(x)u\ge \frac{C}{|x|^{\mu_2(p-1)}},\quad x\in \rn\backslash B_R(0).
\end{align*}
Now we set $\mu_{i+1}:=\mu_i(p-1)-2s$, $i\ge 2$, i.e.,
$$\mu_{i}=(p-1)^{i-2}\Big(\mu_{2}-\frac{2s}{p-2}\Big)+\frac{2s}{p-2},\quad i\ge2.$$
Since $p-1>1$ and $\mu_2<N-2s<2s/(p-2)$, it holds that $\mu_{i+1}<\mu_{i}<N-2s$ for $i\ge2$ and $\mu_i\to-\wq$ as $i\to\wq$.

Fix $i\ge2$ such that $\mu_i>\frac{N-2s}{2}, \mu_{i+1}>\frac{N-2s}{2}$. We claim that there exist constants $C_{i},C_{i+1}>0$ such that
\begin{align}\label{dsn}
  u\ge C_{i+1}w_{\mu_{i+1}}\ \mathrm{if}\ u\ge C_{i}w_{\mu_{i}}.
\end{align}
Indeed, if $u\ge C_iw_{\mu_i}$, then \eq{wqj} implies
\begin{align*}
  (-\De)^s u+V(x)u\ge \frac{C}{|x|^{\mu_{i}(p-1)}},\quad x\in \rn\backslash B_1(0).
\end{align*}
On the other hand, Proposition \r{tb} and Proposition \r{cc5} (iii) yields that $w_{\mu_{i+1}}$ weakly satisfies
\begin{align*}
  (-\De)^s w_{\mu_{i+1}}+V(x)w_{\mu_{i+1}}\le \frac{C_{\mu_{i+1}}}{|x|^{\mu_{i+1}+2s}}+\frac{\epsilon}{|x|^{\mu_{i+1}+2s}}\le \frac{C}{|x|^{\mu_i(p-1)}},\quad x\in \rn\backslash B_{R_i}(0)
\end{align*}
for some $R_i>0$. Then the claim \eq{dsn} follows by  Lemma \r{ws8}.

As a result, for any $\mu>\frac{N-2s}{2}$, by finite iteration above, there exists a constant $D_\mu>0$ such that
\begin{align*}
  u\ge D_\mu w_{\mu},\quad x\in\rn.
\end{align*}
Therefore, choosing $\mu>\frac{N-2s}{2}$ such that $\mu p<N$, we get
\begin{align}\label{opw}
  \int_{\rn}u^{p}\ge C\int_{\rn}w_{\mu}^{p}=+\wq.
\end{align}
which contradicts to \eq{sep}. This completes the proof.

\noindent\noindent {\bf  The case under the assumption ($\mathcal{Q}'_2$) with $\omega =2s$}, i.e.
$2<p<2+\frac{2s}{N}$ and $\sup_{x\in\rn}(1+|x|^{2s})V(x)<\wq$.

The assumption $\sup_{x\in\rn}(1+|x|^{2s})V(x)<\wq$ means $V(x)\le \frac{C}{1+|x|^{2s}}$ for some constant $C>0$.
Let $\mu_1=N$. By Proposition \r{tb} and Proposition \r{cc5} (iii), $w_{N}$ weakly satisfies
\begin{align}\label{qy9}
  (-\De)^s w_{N}+V(x)w_{N}\le-\frac{C_{N}\ln|x|}{|x|^{N+2s}}+\frac{C}{|x|^{N+2s}}\le0,\quad x\in \rn\backslash B_{R_1}(0)
\end{align}
for some $R_1>0$. It follows by Lemma \r{ws8} that there exists $C_1>0$ such that
\begin{align*}
  u\ge C_1w_{N}.
\end{align*}
Now we set $\mu_{i+1}:=\mu_i(p-1)-2s$, $i\ge 1$, i.e.,
$$\mu_{i}=(p-1)^{i-1}\Big(N-\frac{2s}{p-2}\Big)+\frac{2s}{p-2},\quad i\ge1.$$
Since $p-1>1$ and $N<2s/(p-2)$, it holds that $\mu_{i+1}<\mu_{i}\le N$ for $i\ge1$ and $\mu_i\to-\wq$ as $i\to\wq$.
Then by the similar iteration process to \eq{dsn}, we conclude that for any $\mu>\frac{N-2s}{2}$, there exists a constant $D_\mu>0$ such that $u\ge D_\mu w_\mu$. Consequently, choosing $\mu>\frac{N-2s}{2}$ such that $\mu p<N$, we get
\begin{align}\label{opw}
  \int_{\rn}u^{p}\ge C\int_{\rn}w_{\mu}^{p}=+\wq,
\end{align}
which contradicts to \eq{sep}. This completes the proof.

\noindent\noindent {\bf  The case under the assumption ($\mathcal{Q}'_2$) with $\omega \in (0, 2s)$}, i.e.
$2<p<2+\frac{\om}{N+2s-\om}$ and $\sup_{x\in\rn}(1+|x|^{\om})V(x)<\wq$ with $\om\in(0,2s)$.

By the assumption $\sup_{x\in\rn}(1+|x|^{\om})V(x)<\wq$, $V(x)\le \frac{C_\om}{1+|x|^{\om}}$ for some constant $C_\om>0$.
Let $\mu_1\in(N+2s-\om,N+2s)$ be a parameter. By Proposition \r{tb} and Proposition \r{cc5} (iii), $w_{\mu_1}$ weakly satisfies
\begin{align}\label{qyq}
  (-\De)^s w_{\mu_1}+V(x)w_{\mu_1}\le-\frac{C_{\mu_1}}{|x|^{N+2s}}+\frac{C_\om}{|x|^{\mu_1+\om}}\le0,\quad x\in \rn\backslash B_{R_1}(0)
\end{align}
for some $R_1>0$. It follows by Lemma \r{ws8} that there exists $C_1>0$ such that
\begin{align*}
  u\ge C_1w_{\mu_1}.
\end{align*}
Recalling \eq{wqj}, we obtain
\begin{align*}
  (-\De)^s u+V(x)u\ge \frac{C}{|x|^{\mu_1(p-1)}},\quad x\in \rn\backslash B_1(0).
\end{align*}
Since $p<2+\frac{\om}{N+2s-\om}$, we choose  $\mu_2\in (N,N+2s-\om)$ and $\mu_1$ close to $N+2s-\om$ such that
\begin{equation}\label{w60}
  N+2s>\mu_2+\om>\mu_1(p-1).
\end{equation}
We deduce by Proposition \r{tb}, Proposition\r{cc5} (iii) and \eq{w60} that $w_{\mu_2}$ weakly satisfies
\begin{align*}
  (-\De)^s w_{\mu_2}+V(x)w_{\mu_2}\le \frac{C_{\mu_2}}{|x|^{N+2s}}+\frac{C_\om}{|x|^{\mu_2+\om}}\le \frac{C}{|x|^{\mu_1(p-1)}},\quad x\in \rn\backslash B_{R_2}(0)
\end{align*}
for some $R_2>0$.
Consequently,  by Lemma \r{ws8}, there exists $C_2>0$ such that
\begin{align*}
  u\ge C_2w_{\mu_2}.
\end{align*}
Now we set $\mu_{i+1}:=\mu_i(p-1)-\om$, $i\ge 2$, i.e.,
$$\mu_{i}=(p-1)^{i-2}\Big(\mu_2-\frac{\om}{p-2}\Big)+\frac{\om}{p-2},\quad i\ge2.$$
Since $p-1>1$ and $\mu_2<N+2s-\om<\frac{\om}{p-2}$, it holds that $\mu_{i+1}<\mu_{i}<N+2s-\om$ for $i\ge2$ and $\mu_i\to-\wq$ as $i\to\wq$.
The remain proof are similar to that in Case 1 by iteration, so we omit the details.

Consequently we complete the proof of Theorem \r{thm1.3}.
\end{proof}

\section{Decay estimates}\label{dey}
In this section, we further study the decay properties for positive weak solutions of \eq{eqs1.1} by some delicate iteration, and then complete the proof of Theorem \r{thm1.5} and Corollary \r{c30}.

\begin{proof}[{\bf} Proof of Theorem \r{thm1.5}]
Without loss of generality, we fix $\va=1$. Let $u\in H^s_{V,1}(\rn)$ be a positive weak solution of
\begin{align}\label{wqj'}
  (-\De)^su+V(x)u=u^{p-1}, \quad x\in\rn.
\end{align}
By regularity estimates as before, we have
$u\in L^\wq(\rn)\cap C(\rn)$.

\noindent{\bf Case 1: $p_*=2+\frac{2s}{N-2s}< p<2_s^*$ and $\sup_{x\in\rn}(1+|x|^{\om})V(x)<\wq$ with $\om>(N-2s)(p-2)>2s$}.

In this case, we can find a constant $C>0$ such that $V(x)\le \frac{C}{1+|x|^{\om}}$ in $\rn$.
Fix any $\mu\in(N-2s,N)$. By Proposition \r{tb} and Proposition \r{cc5} (iii), $w_{\mu}$ weakly satisfies
\begin{align}\label{qya}
  (-\De)^s w_{\mu}+V(x)w_{\mu}\le-\frac{C_{\mu}}{|x|^{\mu+2s}}+\frac{C_\om}{|x|^{\mu+\om}}\le0,\quad x\in \rn\backslash B_{R_1}(0)
\end{align}
for some $R_1>0$. Then by Lemma \r{ws8}, there exists a constant $D_\mu>0$ such that
\begin{align}\label{3th}
  u\ge D_\mu w_{\mu}.
\end{align}
Since $\om>(N-2s)(p-2)$ and $p<2_s^*$, we choose $\mu\in(N-2s,N)$  such that
\begin{equation}\label{e2c}
 N-2s+\om>\mu(p-1),\quad N+2s>\mu(p-1).
\end{equation}
From \eq{3th},
\begin{align*}
  (-\De)^su+V(x)u=u^{p-1}\ge \frac{C}{|x|^{\mu(p-1)}},\quad x\in \rn\backslash B_1(0).
\end{align*}
On the other hand, by Proposition \r{tb}, Proposition \r{cc5} (iii) and \eq{e2c}, $w_{N-2s}$ weakly satisfies
\begin{align*}
  (-\De)^sw_{N-2s}+V(x)w_{N-2s}\le \frac{C_{N-2s}}{|x|^{N+2s}}+\frac{C_\om}{|x|^{N-2s+\om}}\le \frac{C}{|x|^{\mu(p-1)}},\quad x\in \rn\backslash B_{R_2}(0)
\end{align*}
for some $R_2>0$. Then by Lemma \r{ws8}, there exists a constant $C>0$ such that
\begin{align*}
  u\ge C w_{N-2s}.
\end{align*}

\noindent{\bf Case 2: $p_*=2+\frac{2s}{N-2s}<p<2_s^*$ and $\sup_{x\in\rn}(1+|x|^{\om})V(x)<\wq$ with $2s<\om\le(N-2s)(p-2)$}.

It follows by \eq{qya} and Lemma \r{ws8} that for any $\mu>N-2s$,  there exists a constant $D_\mu>0$ such that
\begin{align*}
  u\ge D_\mu w_{\mu}.
\end{align*}

\noindent{\bf Case 3: $p_*=2+\frac{2s}{N}<p<2_s^*$ and $\sup_{x\in\rn}(1+|x|^{2s})V(x)<\wq$.}

It follows from \eq{qy9} and Lemma \r{ws8} that  there exists a constant $C>0$ such that
\begin{align*}
  u\ge C w_{N}.
\end{align*}

\noindent{\bf Case 4: $p_*=2+\frac{\om}{N+2s-\om}<p<2_s^*$ and $\sup_{x\in\rn}(1+|x|^{\om})V(x)<\wq$ with $\om\in[0,2s)$.}

There exists $C>0$ such that $V(x)\le \frac{C_\om}{1+|x|^{\om}}$ in $\rn$. Let $\la>0$ be a parameter.
By Proposition \r{tb} and Proposition \r{cc5} (iii), $v_{\la}(x):=w_{N+2s-\om}(\la x)$ weakly satisfies
\begin{align}\label{qya'}
  (-\De)^s v_{\la}+V(x)v_{\la}=&\la^{2s}(-\De)^s w_{N+2s-\om}(\la x)+V(x)w_{N+2s-\om}(\la x)\nonumber\\
  \le&\la^{2s}\frac{-C_{N+2s-\om}}{|\la x|^{N+2s}}+\frac{C_\om}{|x|^{\om}}\frac{1}{|\la x|^{N+2s-\om}}\nonumber\\
  =&\frac{1}{\la^{N}}\Big(\frac{-C_{N+2s-\om}}{|x|^{N+2s}}+\frac{\la^{\om-2s}C_\om}{|x|^{N+2s}}\Big)\le0,\quad x\in \rn\backslash B_{R_\la}(0)
\end{align}
provided that $\la$ is large enough. It follows by Lemma \r{ws8} that  there exists a constant $C>0$ such that
\begin{align*}
  u\ge C w_{N+2s-\om}(\la x).
\end{align*}

Summing up the above,  we complete the proof of Theorem \r{thm1.5}.
\end{proof}

By means of Theorem \r{thm1.5}, now we continue to prove Corollary \r{c30}.
\begin{proof}[{\it Proof of Corollary \r{c30}}]
For convenience, we discuss in separated cases.

\noindent{\bf Case 1: $p_*=2+\frac{2s}{N-2s}<p<2_s^*$ and  $\om>(N-2s)(p-2)>2s$.}

The proof is directly completed  by \eq{yyy} and Theorem \r{thm1.5} (i).

\noindent{\bf Case 2: $p_*=2+\frac{2s}{N-2s}<p<2_s^*$ and  $2s<\om<(N-2s)(p-2)$.}

Since $\om<(N-2s)(p-2)$, we can choose $\ga<N-2s$ sufficiently close to $N-2s$ such that
\begin{equation}\label{wq9}
  \om+N-2s<\ga(p-1).
\end{equation}
By \eq{yyy}, we have $u_\va\le \frac{C}{1+|x|^\ga}$. As a result,
\begin{align*}
  \va^{2s}(-\De)^su_\va+Vu_\va=u_\va^{p-1}\le \frac{C}{|x|^{\ga(p-1)}}\le \frac{C}{|x|^{\om+N-2s}},\quad |x|\ge1.
\end{align*}
By Proposition \r{tb} and Corollary \r{cc5} (iii), $w_{N-2s}(x)$ weakly satisfies
\begin{align*}
  \va^{2s}(-\De)^sw_{N-2s}+Vw_{N-2s}\ge \frac{C}{|x|^{\om+N-2s}},\quad |x|\ge1.
\end{align*}
It follows from Lemma \r{ws8} that
$$u_\va\le Cw_{N-2s}$$
for some $C>0$. On the other hand, the lower decay estimate is given by  Theorem \r{thm1.5} (ii).

\noindent{\bf Case 3: $p_*=2+\frac{2s}{N}<p<2_s^*$ and  $\om=2s$.}

The conclusion holds directly  by \eq{yyy} and Theorem \r{thm1.5} (iii).

\noindent{\bf Case 4: $p_*=2+\frac{\om}{N+2s-\om}<p<2_s^*$ and $\om\in[0,2s)$.}

Since $p>2+\frac{\om}{N+2s-\om}$, we can choose $\ga<N+2s-\om$ such that $\ga(p-1)> N+2s$. By Theorem \r{thm1.2}, we have
$$0<u_\va\le \frac{C\va^\ga}{|x|^{\ga}},\quad |x|\ge R_\ga.$$
Then
\begin{align*}
  \va^{2s}(-\De)^su_\va+Vu_\va=u_\va^{p-1}\le \frac{C}{|x|^{\ga(p-1)}}\le \frac{C}{|x|^{N+2s}},\quad |x|\ge R_\ga.
\end{align*}
Let $v_{\la}(x):=w_{N+2s-\om}(\la x)$. By Proposition \r{tb} and Corollary \r{cc5} (iii), $v_{\la}(x)$ weakly satisfies
\begin{align}\label{qia}
  \va^{2s}(-\De)^s v_{\la}+V(x)v_{\la}=&\va^{2s}\la^{2s}(-\De)^s w_{N+2s-\om}(\la x)+V(x)w_{N+2s-\om}(\la x)\nonumber\\
  \ge&\va^{2s}\la^{2s}\frac{-\tilde{C}_{N+2s-\om}}{|\la x|^{N+2s}}+\frac{C_1}{|x|^{\om}}\frac{1}{2|\la x|^{N+2s-\om}}\nonumber\\
  =&\frac{1}{\la^{N}}\Big(\frac{-\va^{2s}\tilde{C}_{N+2s-\om}}{|x|^{N+2s}}+\frac{\la^{\om-2s}C_1/2}{|x|^{N+2s}}\Big)\ge \frac{C_\la}{|x|^{N+2s}},\quad |x|\ge B_{R_\la}
\end{align}
for some $C_\la, R_\la>0$ provided that $\la>0$ is small enough. It follows from Lemma \r{ws8} that  there exists a constant $C>0$ such that
\begin{align}\label{3th'}
  u\le C w_{N+2s-\om}(\la x).
\end{align}
On the other hand, the lower decay estimate is a direct consequence  of  Theorem \r{thm1.5} (iv).

Then  the proof of Corollary \r{c30} is completed.
\end{proof}

\vspace{0.2cm}

\appendix

\section{}\label{sok}

\renewcommand{\theequation}{A.\arabic{equation}}
\setcounter{equation}{0}

\vskip 0.2cm
In this appendix,  we give a complete proof of Proposition \r{tb}.

\begin{proof}[Proof of Proposition \r{tb}]
The estimate of $\Fs w_\mu$ base  on the estimate of $\Fs h_\mu$, where $h_\mu=|x|^{-\mu}\ (x\neq0)$. So we first estimate $\Fs h_\mu$.

For any
given $x\in\rn$ such that $|x|>1$, by  changes of variable we have
\begin{align}\label{sj}
\frac{1}{2}\Fs h_\mu
=&\lim_{r\to0}\int_{\rn\setminus B_r(x)}\frac{|x|^{-\mu}-|y|^{-\mu}}{|x-y|^{N+2s}}\d y\quad (\mathrm{set}\ y=|x|y')\nonumber\\
=&\frac{1}{|x|^{\mu+2s}}\lim_{r\to0}\int_{\rn\setminus B_{r}(\vec{e}_1)}\frac{|y'|^{\mu}-1}{|y'|^{\mu}|y'-\vec{e}_1|^{N+2s}}\d y'\nonumber\\
=&\frac{1}{|x|^{\mu+2s}}\lim_{r\to0}\Big(\int_{B_1(0)\setminus B_{r}(\vec{e}_1)}\frac{|y|^{\mu}-1}{|y|^{\mu}|y-\vec{e}_1|^{N+2s}}\d y\nonumber\\
&\qquad\qquad\qquad+\int_{(\rn\setminus B_1(0))\setminus B_{r}(\vec{e}_1)}\frac{|y|^{\mu}-1}{|y|^{\mu}|y-\vec{e}_1|^{N+2s}}\d y\Big)
\end{align}
where $\vec{e}_1:=\frac{x}{|x|}$ is a unit vector.

For the coordinate transformation $y=\frac{y'}{|y'|^2}$ (a inversion of a sphere), we have
$$|y||y'|=1,\ dy=|y'|^{-2N}dy',\ |y-\vec{e}_1||y'|=|y'-\vec{e}_1|,$$
where we have used that
$$|y-\vec{e}_1|^2|y'|^2=(|y|^2-2y\cdot\vec{e}_1+1)|y'|^2=1-2y'\cdot\vec{e}_1+|y'|^2=|y'-\vec{e}_1|^2.$$
It follows that
\begin{align}\label{ssjj}
  &\int_{B_1(0)\setminus B_{r}(\vec{e}_1)}\frac{|y|^{\mu}-1}{|y|^{\mu}|y-\vec{e}_1|^{N+2s}}\d y \nonumber \\
  =&\int_{(\rn\backslash B_1(0))\setminus (B_{r}(\vec{e}_1))^*}\frac{1-|y'|^{\mu}}{|y'|^{N-2s}|y'-\vec{e}_1|^{N+2s}}\d y'  \nonumber \\
  =&\int_{(\rn\backslash B_1(0))\setminus B_{r}(\vec{e}_1))}\frac{1-|y|^{\mu}}{|y|^{N-2s}|y-\vec{e}_1|^{N+2s}}\d y+\int_{B_{r}(\vec{e}_1)\setminus (B_{r}(\vec{e}_1))^*}\frac{1-|y|^{\mu}}{|y|^{N-2s}|y-\vec{e}_1|^{N+2s}}\d y  \nonumber \\
  &-\int_{(B_{r}(\vec{e}_1))^*\setminus B_{r}(\vec{e}_1)}\frac{1-|y|^{\mu}}{|y|^{N-2s}|y-\vec{e}_1|^{N+2s}}\d y,
\end{align}
where
$$(B_{r}(\vec{e}_1))^*:=\Big\{\frac{y}{|y|^2}\mid y\in B_r(\vec{e}_1)\Big\}. $$
Since
$B_{r/2}(\vec{e}_1)\subset (B_{r}(\vec{e}_1))^*\subset B_{3r/2}(\vec{e}_1)\  \mathrm{as}\ r\to0,$
we have that
$$\frac{|1-|y|^{\mu}|}{|y|^{N-2s}|y-\vec{e}_1|^{N+2s}}\le O\big(\frac{1}{r^{N+2s-1}}\big),\quad y\in \big(B_{r}(\vec{e}_1)\setminus (B_{r}(\vec{e}_1))^*\big)\cup \big(B_{r}(\vec{e}_1))^*\setminus B_{r}(\vec{e}_1)\big).$$
On the other hand, we can verify that
$$|B_{r}(\vec{e}_1)\setminus (B_{r}(\vec{e}_1))^*|+|B_{r}(\vec{e}_1))^*\setminus B_{r}(\vec{e}_1)|=O(r^{N+1})\ \mathrm{as}\ r\to0,$$
it follows that
\begin{align*}
  \int_{B_{r}(\vec{e}_1)\setminus (B_{r}(\vec{e}_1))^*}\frac{1-|y|^{\mu}}{|y|^{N-2s}|y-\vec{e}_1|^{N+2s}}\d y=O(r^{2-2s})\to0\ \mathrm{as}\ r\to0,
\end{align*}
\begin{align*}
  \int_{(B_{r}(\vec{e}_1))^*\setminus B_{r}(\vec{e}_1)}\frac{1-|y|^{\mu}}{|y|^{N-2s}|y-\vec{e}_1|^{N+2s}}\d y=O(r^{2-2s})\to0\ \mathrm{as}\ r\to0.
\end{align*}
Substituting the estimates above into \eq{ssjj}, we obtain that

\begin{align}\label{ssjj1}
  &\int_{B_1(0)\setminus B_{r}(\vec{e}_1)}\frac{|y|^{\mu}-1}{|y|^{\mu}|y-\vec{e}_1|^{N+2s}}\d y \nonumber \\
  =&\int_{(\rn\backslash B_1(0))\setminus B_{r}(\vec{e}_1))}\frac{1-|y|^{\mu}}{|y|^{N-2s}|y-\vec{e}_1|^{N+2s}}\d y+O(r^{2-2s})\ \mathrm{as}\ r\to0.
  \end{align}
Putting (\ref {ssjj1}) into (\ref {sj}) yields
\begin{align}\label{ssjj2}
\frac{1}{2}\Fs h_\mu
=&\frac{1}{|x|^{\mu+2s}}\lim_{r\to0}\int_{(\rn\setminus B_1(0))\setminus B_{r}(\vec{e}_1)}\frac{|y|^{\mu}-1}{|y-\vec{e}_1|^{N+2s}}\Big(
\frac{1}{|y|^{\mu}}-\frac{1}{|y|^{N-2s}}\Big)\d y\nonumber\\
=&\frac{1}{|x|^{\mu+2s}}\int_{\rn\setminus B_1(0)}\frac{|y|^{\mu}-1}{|y-\vec{e}_1|^{N+2s}}\Big(
\frac{1}{|y|^{\mu}}-\frac{1}{|y|^{N-2s}}\Big)\d y
:=A_\mu\frac{1}{|x|^{\mu+2s}},
\end{align}
where $\vec{e}_1$ is not a singular point for the last integral in \eq{ssjj2}
since
$$\frac{|y|^{\mu}-1}{|y-\vec{e}_1|^{N+2s}}\Big(
\frac{1}{|y|^{\mu}}-\frac{1}{|y|^{N-2s}}\Big)=O\Big(\frac{1}{|x-\vec{e}_1|^{N+2s-2}}\Big)\ \ \mathrm{as}\ y\to \vec{e}_1.$$
Noting the asymptotic behavior of $\frac{|y|^{\mu}-1}{|y-\vec{e}_1|^{N+2s}}(
\frac{1}{|y|^{\mu}}-\frac{1}{|y|^{N-2s}})$ as $|y|\to\wq$,
it is easy to check that $A_\mu$  satisfies
\begin{equation}\label{sj0}
\begin{aligned}
\left\{
  \begin{array}{ll}
   A_\mu\in(0,+\wq),& \mathrm{if }\ 0<\mu<N-2s,\vspace{1mm}\\
   A_\mu=0, & \mathrm{if }\ \mu=N-2s;\vspace{1mm}\\
A_\mu\in(-\wq,0),& \mathrm{if }\ N-2s<\mu<N\vspace{1mm}\\
 A_\mu=-\wq,&\mathrm{if}\ \mu\ge N.
  \end{array}
\right.
\end{aligned}
\end{equation}

Now we are ready to estimate $\Fs w_{\mu}$ according to different cases stated in \eqref{sj0}.

\textbf{Case 1.} {\bf $\mu\in (0,N)\setminus\{N-2s\}$. }

By   changing   variable as in \eq{sj}, we have
\begin{equation}\label{sj1}
\begin{aligned}
&\frac{1}{2}\left|\Fs w_\mu-\Fs h_\mu\right|\\
\quad\le&\frac{1}{|x|^{\mu+2s}}\Big|\lim_{r\to0}\int_{\rn\setminus B_{r}(\vec{e}_1)}\frac{(|x|^{-2}+1)^{-\frac{\mu}{2}}-(|x|^{-2}+|y|^2)^{-\frac{\mu}{2}}-1+|y|^{-\mu}}{|y-\vec{e}_1|^{N+2s}}\d y\Big|\\
:=&\frac{1}{|x|^{\mu+2s}}\Big|\lim_{r\to0}\int_{\rn\setminus B_{r}(\vec{e}_1)}\frac{L(x,y)}{|y-\vec{e}_1|^{N+2s}}\d y\Big|.
\end{aligned}
\end{equation}
For any $M>2$ and $\rho'\in(r,1/2)$,  letting $|x|>2M$, we have
\begin{align}\label{wcl}
&\Big|\lim_{r\to0}\int_{\rn\setminus B_{r}(\vec{e}_1)}\int_{\rn}\frac{L(x,y)}{|y-\vec{e}_1|^{N+2s}}\d y\Big|\nonumber\\
\le&\int_{|y-\vec{e}_1|>M}\frac{|L(x,y)|}{|y-\vec{e}_1|^{N+2s}}\d y+\Big|\lim_{r\to0}\int_{r<|y-\vec{e}_1|<\rho'}\frac{L(x,y)}{|y-\vec{e}_1|^{N+2s}}\d y\Big|\nonumber\\
&+\int_{\{\rho'\le|y-\vec{e}_1|\le M\}\cap\{|y|\ge\rho'\}}\frac{|L(x,y)|}{|y-\vec{e}_1|^{N+2s}}\d y
+\int_{\{\rho'\le|y-\vec{e}_1|\le M\}\cap\{|y|<\rho'\}}\frac{|L(x,y)|}{|y-\vec{e}_1|^{N+2s}}\d y.
\end{align}
Clearly, $|L(x,y)|\le 4$ for $|y-\vec{e}_1|>M$, then
\begin{align}\label{wcn}
  \int_{|y-\vec{e}_1|>M}\frac{|L(x,y)|}{|y-\vec{e}_1|^{N+2s}}\d y\le 4\int_{|y-\vec{e}_1|>M}\frac{1}{|y-\vec{e}_1|^{N+2s}}\d y\le \frac{C}{M^{2s}}.
\end{align}
For $|y-\vec{e}_1|<\rho'$, we have $1/2 \le|y|\le 3/2$, then by Taylor expansion, we have
\begin{align}\label{ooe}
  |y|^{-\mu}-|\vec{e}_1|=-\mu \vec{e}_1\cdot(y-\vec{e}_1)+O(|y-\vec{e}_1|^2),
\end{align}
$$(|x|^{-2}+|y|^2)^{-\frac{\mu}{2}}-(|x|^{-2}+|\vec{e}_1|^2)^{-\frac{\mu}{2}}=
-\mu(|x|^{-2}+1)^{-\frac{\mu}{2}-1}\vec{e}_1\cdot(y-\vec{e}_1)+O(|y-\vec{e}_1|^2). $$
By symmetry,
\begin{align}\label{nnm}
  \int_{r<|y-\vec{e}_1|<\rho'}\frac{\vec{e}_1\cdot(y-\vec{e}_1)}{|y-\vec{e}_1|^{N+2s}}\d y=0.
\end{align}
Therefore,
\begin{align}\label{wxv}
  \Big|\lim_{r\to0}\int_{r<|y-\vec{e}_1|<\rho'}\frac{L(x,y)}{|y-\vec{e}_1|^{N+2s}}\d y\Big|
  \le C\lim_{r\to0}\int_{r<|y-\vec{e}_1|<\rho'}\frac{1}{|y-\vec{e}_1|^{N+2s-2}}\d y\le C(\rho')^{2-2s}.
\end{align}
For $y\in\{\rho'\le|y-\vec{e}_1|\le M\}\cap\{|y|\ge\rho'\}$, we have
$$|(|x|^{-2}+1)^{-\mu}-1|\le C|x|^{-2},$$
$$\big|(|x|^{-2}+|y|^2)^{-\frac{\mu}{2}}-|y|^{-\mu}\big|\le \frac{C}{(\rho')^{\mu+1}}|x|^{-2},$$
and thereby,
\begin{align}\label{osx}
  \int_{\{\rho'\le|y-\vec{e}_1|\le M\}\cap\{|y|\ge\rho'\}}\frac{|L(x,y)|}{|y-\vec{e}_1|^{N+2s}}\d y
  \le&C|x|^{-2}\Big(1+\frac{C}{(\rho')^{\mu+1}}\Big)\int_{\{\rho'\le|y-\vec{e}_1|\}}\frac{1}{|y-\vec{e}_1|^{N+2s}}\d y\nonumber\\
  \le&C|x|^{-2}\Big(1+\frac{C}{(\rho')^{\mu+1}}\Big)\frac{1}{(\rho')^{2s}}.
\end{align}
For $y\in\{\rho'\le|y-\vec{e}_1|\le M\}\cap\{|y|<\rho'\}$, then $|y-\vec{e}_1|\ge 1-|y|\ge 1/2$ and
$$|L(x,y)|=1-(|x|^{-2}+1)^{-\frac{\mu}{2}}+|y|^{-\mu}-(|x|^{-2}+|y|^2)^{-\frac{\mu}{2}}\le 1+|y|^{-\mu},$$
and consequently,
\begin{align}\label{onn}
  \int_{\{\rho'\le|y-\vec{e}_1|\le M\}\cap\{|y|<\rho'\}}\frac{|L(x,y)|}{|y-\vec{e}_1|^{N+2s}}\d y\le &\frac{1}{(1/2)^{N+2s}}\int_{\{|y|<\rho'\}}(1+|y|^{-\mu})\d y\nonumber\\
  =&C(\rho')^{N}+C(\rho')^{N-\mu}.
\end{align}
As a result, we conclude from \eq{wcl}-\eq{wcn} and \eq{wxv}-\eq{onn} that
\begin{align*}
  &\Big|\lim_{r\to0}\int_{\rn\setminus B_{r}(\vec{e}_1)}\frac{L(x,y)}{|y-\vec{e}_1|^{N+2s}}\d y\Big|\nonumber\\
  \le&C\Big(\frac{1}{M^{2s}}+(\rho')^{2-2s}+(\rho')^{N}+(\rho')^{N-\mu}+|x|^{-2}\big(1+\frac{1}{(\rho')^{\mu+1}}\big)\frac{1}{(\rho')^{2s}}\Big)
\end{align*}
for a constant $C>0$ independent of $M>2$ and $\rho'\in(0,1/2)$.
Letting $M\to+\wq$ and $\rho'\to0_+$, we have
\begin{align}\label{yfg}
 &\lim_{|x|\to\wq}\Big|\lim_{r\to0}\int_{\rn\setminus B_{r}(\vec{e}_1)}\frac{L(x,y)}{|y-\vec{e}_1|^{N+2s}}\d y\Big|\nonumber\\
 \le&\mathop{\lim}_{M\to+\wq\atop{\rho'\to0_+}}\lim_{|x|\to\wq}C\Big(\frac{1}{M^{2s}}+(\rho')^{2-2s}+(\rho')^{N}+(\rho')^{N-\mu}
 +|x|^{-2}\big(1+\frac{1}{(\rho')^{\mu+1}}\big)\frac{1}{(\rho')^{2s}}\Big)\nonumber\\
 =&0.
\end{align}
Then there exists $R_\mu>0$ such that
\begin{equation}\label{sj3}
\begin{aligned}
\Big|\lim_{r\to0}\int_{\rn\setminus B_{r}(\vec{e}_1)}\frac{L(x,y)}{|y-\vec{e}_1|^{N+2s}}\d y\Big|\le \frac{1}{2}|A_\mu|,\ \ |x|>R_\mu.
\end{aligned}
\end{equation}
Putting $\eq{sj}$--$\eq{sj1}$ and $\eq{sj3}$ together, we infer that
\begin{align}
&0<\frac{A_\mu}{2}\frac{1}{|x|^{\mu+2s}}\le\frac{1}{2}\Fs w_\mu\le\frac{3A_\mu}{2}\frac{1}{|x|^{\mu+2s}},\ \mathrm{if }\ |x|>R_\mu\ \mathrm{and}\ \mu\in(0,N-2s);\nonumber\\
&\frac{3A_\mu}{2}\frac{1}{|x|^{\mu+2s}}\le\frac{1}{2}\Fs w_\mu\le\frac{A_\mu}{2}\frac{1}{|x|^{\mu+2s}}<0,\ \mathrm{if }\ |x|>R_\mu\ \mathrm{and}\ \mu\in(N-2s,N).\nonumber
\end{align}

\textbf{Case 2.} $\mu\ge N$.

Also by changing variable as in \eq{sj}, there holds
\begin{align}\label{sj5}
\frac{1}{2}\Fs w_\mu=&\frac{1}{|x|^{\mu+2s}}\lim_{r\to0}\int_{\rn\setminus B_{r}(\vec{e}_1)}\frac{(|x|^{-2}+1)^{-\frac{\mu}{2}}-(|x|^{-2}+|y|^2)^{-\frac{\mu}{2}}}{|y-\vec{e}_1|^{N+2s}}\d y\nonumber\\
=&\frac{1}{|x|^{\mu+2s}}\Big(\lim_{r\to0}\int_{(\rn\setminus B_{1/2}(0))\backslash B_{r}(\vec{e}_1) }\frac{(|x|^{-2}+1)^{-\frac{\mu}{2}}-(|x|^{-2}+|y|^2)^{-\frac{\mu}{2}}}{|y-\vec{e}_1|^{N+2s}}\d y\nonumber\\
&\ +\int_{B_{1/2}(0)}\frac{(|x|^{-2}+1)^{-\frac{\mu}{2}}}{|y-\vec{e}_1|^{N+2s}}\d y
-\int_{B_{1/2}(0)}\frac{(|x|^{-2}+|y|^2)^{-\frac{\mu}{2}}}{|y-\vec{e}_1|^{N+2s}}\d y\Big).
\end{align}
Same as \eq{sj1}, we denote $L(x,y):=(|x|^{-2}+1)^{-\frac{\mu}{2}}-(|x|^{-2}+|y|^2)^{-\frac{\mu}{2}}-1+|y|^{-\mu}$.
For any $M>2$ and $\rho'\in(r,1/2)$,  letting $|x|>2M$, be the same arguments as \r{wcn}, \eq{wxv} and \eq{osx}, we have
\begin{align}\label{wcl}
&\Big|\lim_{r\to0}\int_{(\rn\setminus B_{1/2}(0))\backslash B_{r}(\vec{e}_1) }\frac{L(x,y)}{|y-\vec{e}_1|^{N+2s}}\d y\Big|\nonumber\\
\le&\int_{|y-\vec{e}_1|>M}\frac{|L(x,y)|}{|y-\vec{e}_1|^{N+2s}}\d y+\Big|\lim_{r\to0}\int_{r<|y-\vec{e}_1|<\rho'}\frac{L(x,y)}{|y-\vec{e}_1|^{N+2s}}\d y\Big|\nonumber\\
&+\int_{\{\rho'\le|y-\vec{e}_1|\le M\}\cap\{|y|\ge\frac{1}{2}\}}\frac{|L(x,y)|}{|y-\vec{e}_1|^{N+2s}}\d y\nonumber\\
\le&\frac{C}{M^{2s}}+C(\rho')^{2-2s}+C|x|^{-2}\frac{1}{(\rho')^{2s}},
\end{align}
which implies that
\begin{align}\label{sj6}
 &\lim_{|x|\to\wq}\Big|\lim_{r\to0}\int_{(\rn\setminus B_{1/2}(0))\backslash B_{r}(\vec{e}_1) }\frac{L(x,y)}{|y-\vec{e}_1|^{N+2s}}\d y\Big|\nonumber\\
 \le&\mathop{\lim}_{M\to+\wq\atop{\rho'\to0_+}}\lim_{|x|\to\wq}C\Big(\frac{1}{M^{2s}}+(\rho')^{2-2s}+|x|^{-2}\frac{1}{(\rho')^{2s}}\Big)=0.
\end{align}

In view of \eq{ooe} and \eq{nnm},  the following integral converges to a constant independent of $x$ as $r\to0$,
$$\lim_{r\to0}\int_{(\rn\setminus B_{1/2}(0))\backslash B_{r}(\vec{e}_1)}\frac{1-|y|^{-\mu}}{|y-\vec{e}_1|^{N+2s}}\d y:=C^{*}\in \R,$$
and thereby from \eq{sj6},
\begin{align}\label{evk}
 &\lim_{|x|\to\wq}\lim_{r\to0}\int_{(\rn\setminus B_{1/2}(0))\backslash  B_{r}(\vec{e}_1) }\frac{(|x|^{-2}+1)^{-\frac{\mu}{2}}-(|x|^{-2}+|y|^2)^{-\frac{\mu}{2}}}{|y-\vec{e}_1|^{N+2s}}\d y=C^*.
\end{align}
Obviously,
\begin{align}\label{sj7}
\int_{B_{1/2}(0)}\frac{(|x|^{-2}+1)^{-\frac{\mu}{2}}}{|y-\vec{e}_1|^{N+2s}}\d y\le 2^{2s}\omega_N(|x|^{-2}+1)^{-\frac{\mu}{2}}\le2^{2s}\omega_N.
\end{align}
Letting $|x|>4$, we have
\begin{align}\label{sj8}
&\int_{B_{1/2}(0)}\frac{(|x|^{-2}+|y|^2)^{-\frac{\mu}{2}}}{|y-\vec{e}_1|^{N+2s}}\d y\nonumber\\
\le&2^{N+2s}\int_{B_{1/2}(0)}(|x|^{-2}+|y|^2)^{-\frac{\mu}{2}}\d y
=2^{N+2s}|x|^{\mu-N}\int_{B_{|x|/2}(0)}\frac{1}{(1+|y|^2)^{\frac{\mu}{2}}}\d y\nonumber\\
\le&2^{N+2s}|x|^{\mu-N}\Big(\int_{B_1(0)}\frac{1}{(1+|y|^2)^{\frac{\mu}{2}}}+\omega_N\int^{|x|/2}_1\frac{1}{r^{\mu-N+1}}\d r\Big)\nonumber\\
\le&\left\{\begin{array}{ll}
2^{N+2s}(\omega_N\ln|x|+C),& \mu=N, \vspace{1mm}\\
 2^{N+2s}(\frac{\omega_N}{\mu-N}+C)|x|^{\mu-N},& \mu>N,
  \end{array} \right.
\end{align}
where $\om_N:=\int_{\partial B_1(0)}dS$.
On the other hand,
\begin{align}\label{sj9}
&\int_{B_{1/2}(0)}\frac{(|x|^{-2}+|y|^2)^{-\frac{\mu}{2}}}{|y-\vec{e}_1|^{N+2s}}\d y\nonumber\\
\ge&(1/2)^{N+2s}\int_{B_{1/2}(0)}(|x|^{-2}+|y|^2)^{-\frac{\mu}{2}}\d y=(1/2)^{N+2s}|x|^{\mu-N}\int_{B_{|x|/2}(0)}\frac{1}{(1+|y|^2)^{\frac{\mu}{2}}}\d y\nonumber\\
\ge&(1/2)^{N+2s}|x|^{\mu-N}\Big(\int_{B_1(0)}\frac{1}{(1+|y|^2)^{\frac{\mu}{2}}}+\frac{\omega_N}{2^\mu}\int^{|x|/2}_1\frac{1}{r^{\mu-N+1}}\d r\Big)\nonumber\\
\ge&\left\{\begin{array}{ll}
(1/2)^{N+2s}(\frac{\omega_N}{2^N}\ln|x|-C),& \mu=N, \vspace{1mm}\\
\frac{N\omega_N}{2^{2N+2s}}|x|^{\mu-N},& \mu>N.
  \end{array} \right.
\end{align}
Summing up the estimates $\eq{sj5}$--$\eq{sj9}$ above, there exists $R_\mu>0$ and $\tilde{C}_1,\tilde{C}_2, \tilde{C}_3, \tilde{C}_4>0$ such that
\begin{equation*}\label{sa}
\begin{aligned}
&-\frac{\tilde{C}_2\ln|x|}{|x|^{N+2s}}\le\frac{1}{2}\Fs w_\mu\le-\frac{\tilde{C}_1\ln|x|}{|x|^{N+2s}}<0,\ \mathrm{if }\ |x|>R_\mu\ \mathrm{and}\ \mu=N;\\
&-\frac{\tilde{C}_4}{|x|^{N+2s}}\le\frac{1}{2}\Fs w_\mu\le-\frac{\tilde{C}_3}{|x|^{N+2s}}<0,\ \mathrm{if }\ |x|>R_\mu\ \mathrm{and}\ \mu>N,
\end{aligned}
\end{equation*}
where
$$\tilde{C}_1=\frac{\omega_N}{2^{2N+2s+1}},\, \tilde{C}_2=2^{N+2s+1}\omega_N,\, \tilde{C}_3=\frac{\omega_N}{2^{2N+2s+1}},\,\tilde{C}_4=\frac{2^{N+2s+1}}{\mu-N}\omega_N.$$

\textbf{Case 3.} $\mu=N-2s$.

In this case, $w_\mu=(1+|x|^2)^{-\frac{N-2s}{2}}$ is the fundamental solution of the critical fractional equation
$$\Fs u=C_{N-2s}u^{2_s^*-1}$$
for some constant $C_{N-2s}>0$ (see \cite{clo}).

As a consequence,  the proof is completed.
\end{proof}

%\textbf{Data Availability Statement:} Data sharing is not applicable to this article as no new
%data were created or analyzed in this study.

\end{document}